\documentclass[a4paper,10pt, oneside]{amsart}
\usepackage[utf8]{inputenc}
\usepackage{url}
\usepackage[hidelinks]{hyperref}
\usepackage{pdfpages}
\usepackage[DIV=10]{typearea}

\usepackage{microtype}
\usepackage{amssymb,amsmath,amsopn,amsxtra,amsthm,amsfonts}
\usepackage{MnSymbol}
\usepackage[mathcal]{euscript}
\let\mathscr\mathcal
\usepackage[bb=px]{mathalfa}

\usepackage[T1]{fontenc}

\newcommand{\yo}{\text{\usefont{U}{min}{m}{n}\symbol{'210}}}

\DeclareFontFamily{U}{min}{}
\DeclareFontShape{U}{min}{m}{n}{<-> udmj30}{}

\usepackage{enumitem}
\setlist[enumerate,1]{label={(\arabic*)},itemsep=\parskip} 
\setlist[itemize,1]{itemsep=\parskip} 
\newlist{thmlist}{enumerate}{2}
\setlist[thmlist,1]{label={\em(\roman*)},ref={(\roman*)},%
  itemsep=\parskip,leftmargin=*,align=left}
\setlist[thmlist,2]{label={\em(\alph*)},ref={(\alph*)},%
  itemsep=\parskip,leftmargin=*,align=left,topsep=0.1cm}
\newlist{defnlist}{enumerate}{2}
\setlist[defnlist,1]{label={(\roman*)},ref={(\roman*)},itemsep=\parskip,%
  leftmargin=*,align=left}
\setlist[defnlist,2]{label={(\alph*)},ref={(\alph*)},itemsep=\parskip,%
  leftmargin=*,align=left,topsep=0.1cm}

 \newtheorem*{thm*}{Theorem}

\newtheorem{thm}[subsection]{Theorem}
\newtheorem{cor}[subsection]{Corollary}
\newtheorem{lem}[subsection]{Lemma}
\newtheorem{prop}[subsection]{Proposition}

\theoremstyle{definition}
\newtheorem{defn}[subsection]{Definition}
\newtheorem{obs}[subsection]{Observation}
\newtheorem{rem}[subsection]{Remark}
\newtheorem{warning}[subsection]{Warning}

\newtheorem{exam}[subsection]{Example}

\newtheorem{notation}[subsection]{Notation}

\renewcommand{\eqref}[1]{(\ref{#1})}

\numberwithin{equation}{subsection}

\usepackage{tikz}
\usetikzlibrary{matrix}
\usepackage{tikz-cd}

\usepackage{xspace}

\usepackage{xifthen}

\usepackage{xparse}

\usepackage{mathtools}

\newcommand{\nc}{\newcommand}
\nc{\renc}{\renewcommand}
\nc{\ssec}{\subsection}
\nc{\sssec}{\subsubsection}
\nc{\on}{\operatorname}
\nc{\term}[1]{#1\xspace}

\nc{\sA}{\ensuremath{\mathcal{A}}\xspace}
\nc{\sB}{\ensuremath{\mathcal{B}}\xspace}
\nc{\sC}{\ensuremath{\mathcal{C}}\xspace}
\nc{\sD}{\ensuremath{\mathcal{D}}\xspace}
\nc{\sE}{\ensuremath{\mathcal{E}}\xspace}
\nc{\sF}{\ensuremath{\mathcal{F}}\xspace}
\nc{\sG}{\ensuremath{\mathcal{G}}\xspace}
\nc{\sH}{\ensuremath{\mathcal{H}}\xspace}
\nc{\sI}{\ensuremath{\mathcal{I}}\xspace}
\nc{\sJ}{\ensuremath{\mathcal{J}}\xspace}
\nc{\sK}{\ensuremath{\mathcal{K}}\xspace}
\nc{\sL}{\ensuremath{\mathcal{L}}\xspace}
\nc{\sM}{\ensuremath{\mathcal{M}}\xspace}
\nc{\sN}{\ensuremath{\mathcal{N}}\xspace}
\nc{\sO}{\ensuremath{\mathcal{O}}\xspace}
\nc{\sP}{\ensuremath{\mathcal{P}}\xspace}
\nc{\sQ}{\ensuremath{\mathcal{Q}}\xspace}
\nc{\sR}{\ensuremath{\mathcal{R}}\xspace}
\nc{\sS}{\ensuremath{\mathcal{S}}\xspace}
\nc{\sT}{\ensuremath{\mathcal{T}}\xspace}
\nc{\sU}{\ensuremath{\mathcal{U}}\xspace}
\nc{\sV}{\ensuremath{\mathcal{V}}\xspace}
\nc{\sW}{\ensuremath{\mathcal{W}}\xspace}
\nc{\sX}{\ensuremath{\mathcal{X}}\xspace}
\nc{\sY}{\ensuremath{\mathcal{Y}}\xspace}
\nc{\sZ}{\ensuremath{\mathcal{Z}}\xspace}

\nc{\bA}{\ensuremath{\mathbf{A}}\xspace}
\nc{\bB}{\ensuremath{\mathbf{B}}\xspace}
\nc{\bC}{\ensuremath{\mathbf{C}}\xspace}
\nc{\bD}{\ensuremath{\mathbf{D}}\xspace}
\nc{\bE}{\ensuremath{\mathbf{E}}\xspace}
\nc{\bF}{\ensuremath{\mathbf{F}}\xspace}
\nc{\bG}{\ensuremath{\mathbf{G}}\xspace}
\nc{\bH}{\ensuremath{\mathbf{H}}\xspace}
\nc{\bI}{\ensuremath{\mathbf{I}}\xspace}
\nc{\bJ}{\ensuremath{\mathbf{J}}\xspace}
\nc{\bK}{\ensuremath{\mathbf{K}}\xspace}
\nc{\bL}{\ensuremath{\mathbf{L}}\xspace}
\nc{\bM}{\ensuremath{\mathbf{M}}\xspace}
\nc{\bN}{\ensuremath{\mathbf{N}}\xspace}
\nc{\bO}{\ensuremath{\mathbf{O}}\xspace}
\nc{\bP}{\ensuremath{\mathbf{P}}\xspace}
\nc{\bQ}{\ensuremath{\mathbf{Q}}\xspace}
\nc{\bR}{\ensuremath{\mathbf{R}}\xspace}
\nc{\bS}{\ensuremath{\mathbf{S}}\xspace}
\nc{\bT}{\ensuremath{\mathbf{T}}\xspace}
\nc{\bU}{\ensuremath{\mathbf{U}}\xspace}
\nc{\bV}{\ensuremath{\mathbf{V}}\xspace}
\nc{\bW}{\ensuremath{\mathbf{W}}\xspace}
\nc{\bX}{\ensuremath{\mathbf{X}}\xspace}
\nc{\bY}{\ensuremath{\mathbf{Y}}\xspace}
\nc{\bZ}{\ensuremath{\mathbf{Z}}\xspace}

\nc{\dA}{\ensuremath{\mathds{A}}\xspace}
\nc{\dB}{\ensuremath{\mathds{B}}\xspace}
\nc{\dC}{\ensuremath{\mathds{C}}\xspace}
\nc{\dD}{\ensuremath{\mathds{D}}\xspace}
\nc{\dE}{\ensuremath{\mathds{E}}\xspace}
\nc{\dF}{\ensuremath{\mathds{F}}\xspace}
\nc{\dG}{\ensuremath{\mathds{G}}\xspace}
\nc{\dH}{\ensuremath{\mathds{H}}\xspace}
\nc{\dI}{\ensuremath{\mathds{I}}\xspace}
\nc{\dJ}{\ensuremath{\mathds{J}}\xspace}
\nc{\dK}{\ensuremath{\mathds{K}}\xspace}
\nc{\dL}{\ensuremath{\mathds{L}}\xspace}
\nc{\dM}{\ensuremath{\mathds{M}}\xspace}
\nc{\dN}{\ensuremath{\mathds{N}}\xspace}
\nc{\dO}{\ensuremath{\mathds{O}}\xspace}
\nc{\dP}{\ensuremath{\mathds{P}}\xspace}
\nc{\dQ}{\ensuremath{\mathds{Q}}\xspace}
\nc{\dR}{\ensuremath{\mathds{R}}\xspace}
\nc{\dS}{\ensuremath{\mathds{S}}\xspace}
\nc{\dT}{\ensuremath{\mathds{T}}\xspace}
\nc{\dU}{\ensuremath{\mathds{U}}\xspace}
\nc{\dV}{\ensuremath{\mathds{V}}\xspace}
\nc{\dW}{\ensuremath{\mathds{W}}\xspace}
\nc{\dX}{\ensuremath{\mathds{X}}\xspace}
\nc{\dY}{\ensuremath{\mathds{Y}}\xspace}
\nc{\dZ}{\ensuremath{\mathds{Z}}\xspace}

\nc{\bbA}{\ensuremath{\mathbb{A}}\xspace}
\nc{\bbB}{\ensuremath{\mathbb{B}}\xspace}
\nc{\bbC}{\ensuremath{\mathbb{C}}\xspace}
\nc{\bbD}{\ensuremath{\mathbb{D}}\xspace}
\nc{\bbE}{\ensuremath{\mathbb{E}}\xspace}
\nc{\bbF}{\ensuremath{\mathbb{F}}\xspace}
\nc{\bbG}{\ensuremath{\mathbb{G}}\xspace}
\nc{\bbH}{\ensuremath{\mathbb{H}}\xspace}
\nc{\bbI}{\ensuremath{\mathbb{I}}\xspace}
\nc{\bbJ}{\ensuremath{\mathbb{J}}\xspace}
\nc{\bbK}{\ensuremath{\mathbb{K}}\xspace}
\nc{\bbL}{\ensuremath{\mathbb{L}}\xspace}
\nc{\bbM}{\ensuremath{\mathbb{M}}\xspace}
\nc{\bbN}{\ensuremath{\mathbb{N}}\xspace}
\nc{\bbO}{\ensuremath{\mathbb{O}}\xspace}
\nc{\bbP}{\ensuremath{\mathbb{P}}\xspace}
\nc{\bbQ}{\ensuremath{\mathbb{Q}}\xspace}
\nc{\bbR}{\ensuremath{\mathbb{R}}\xspace}
\nc{\bbS}{\ensuremath{\mathbb{S}}\xspace}
\nc{\bbT}{\ensuremath{\mathbb{T}}\xspace}
\nc{\bbU}{\ensuremath{\mathbb{U}}\xspace}
\nc{\bbV}{\ensuremath{\mathbb{V}}\xspace}
\nc{\bbW}{\ensuremath{\mathbb{W}}\xspace}
\nc{\bbX}{\ensuremath{\mathbb{X}}\xspace}
\nc{\bbY}{\ensuremath{\mathbb{Y}}\xspace}
\nc{\bbZ}{\ensuremath{\mathbb{Z}}\xspace}

\nc{\mrm}[1]{\ensuremath{\mathrm{#1}}\xspace}
\nc{\mit}[1]{\ensuremath{\mathit{#1}}\xspace}
\nc{\mbf}[1]{\ensuremath{\mathbf{#1}}\xspace}
\nc{\mcal}[1]{\ensuremath{\mathcal{#1}}\xspace}
\nc{\msc}[1]{\ensuremath{\mathscr{#1}}\xspace}
\nc{\mfr}[1]{\ensuremath{\mathfrak{#1}}\xspace}

\renc{\bar}[1]{\overline{#1}}
\nc{\sub}{\subset}
\nc{\too}{\longrightarrow}
\nc{\hook}{\hookrightarrow}
\nc*{\hooklongrightarrow}{\ensuremath{\lhook\joinrel\relbar\joinrel\rightarrow}}
\nc{\hooklong}{\hooklongrightarrow}
\nc{\twoheadlongrightarrow}{\relbar\joinrel\twoheadrightarrow}
\nc{\shiso}{\approx}
\nc{\isoto}{\xrightarrow{\sim}}
\nc{\isofrom}{\xleftarrow{\sim}}
\renc{\ge}{\geqslant}
\renc{\le}{\leqslant}

\nc{\id}{\mathrm{id}}

\DeclareMathOperator{\Hom}{\on{Hom}}
\nc{\uHom}{\underline{\smash{\Hom}}}
\DeclareMathOperator{\Maps}{\on{Map}}

\DeclareMathOperator{\End}{\on{End}}

\nc{\uEnd}{\underline{\smash{\End}}}

\renc{\lim}{\varprojlim}
\makeatletter
\newcommand{\colim@}[2]{%
  \vtop{\m@th\ialign{##\cr
    \hfil$#1\operator@font colim$\hfil\cr
    \noalign{\nointerlineskip\kern1.5\ex@}#2\cr
    \noalign{\nointerlineskip\kern-\ex@}\cr}}%
}
\newcommand{\colim}{%
  \mathop{\mathpalette\colim@{\rightarrowfill@\textstyle}}\nmlimits@
}
\makeatother

\nc{\Cofib}{\on{Cofib}}
\nc{\Fib}{\on{Fib}}
\nc{\initial}{\varnothing}
\nc{\op}{\mathrm{op}}

\nc{\Set}{\mrm{Set}}
\nc{\Spc}{\mrm{Spc}}
\nc{\Spt}{\mrm{Spt}}
\nc{\Spec}{\on{Spec}}
\nc{\Stk}{\mrm{Stk}}
\nc{\Sch}{\mrm{Sch}}
\nc{\aff}{\mrm{aff}}
\nc{\A}{\mbf{A}}
\renc{\P}{\mbf{P}}
\nc{\cl}{{\mrm{cl}}}
\nc{\bDelta}{\mathbf{\Delta}}
\nc{\un}{\mathbf{1}}
\nc{\Tot}{\on{Tot}}
\nc{\Cech}{\textnormal{\v{C}}}
\nc{\Mod}{\mrm{Mod}}
\nc{\Qcoh}{\on{Qcoh}}
\nc{\free}{\mrm{free}}
\nc{\perf}{\mrm{perf}}
\nc{\aperf}{\mrm{aperf}}
\nc{\coh}{\mrm{coh}}
\nc{\Cat}{\mrm{Cat}}
\nc{\unitm}{\mbf{1}}
\nc{\sphere}{\mbf{S}}
\nc{\Z}{\mbf{Z}}
\nc{\Map}{\mrm{Map}}
\nc{\map}{\mrm{map}}
\nc{\PrL}{\mathcal{P}r^\mrm{L}}
\nc{\PrR}{\mathcal{P}r^\mrm{R}}
\nc{\PrLst}{\mathcal{P}r^\mrm{L}_\mrm{St}}
\nc{\Motadd}{\mathcal{M}_\mrm{add}}
\nc{\Motnc}{\mathcal{M}_{\mrm{loc}}}
\nc{\Einfty}{{\sE_\infty}}
\nc{\E}[1]{{\sE_{#1}}}
\nc{\modmod}{/\!\!/}
\nc{\heart}{\heartsuit}
\nc{\proj}{\mrm{proj}}
\nc{\LL}{\on{L}}
\nc{\K}{\on{K}}
\nc{\G}{\on{G}}
\nc{\GL}{\on{GL}}
\nc{\BGL}{\on{BGL}}
\nc{\M}{\on{M}}
\nc{\KH}{\on{KH}}
\nc{\Alg}{\on{Alg}}
\nc{\CAlg}{\on{CAlg}}
\nc{\cn}{\mrm{cn}}
\nc{\hw}{\mrm{Hw}}
\nc{\htt}{\mrm{Ht}}
\nc{\Fun}{\on{Fun}}
\nc{\Funadd}{\on{Fun}_{\mrm{add}}}
\nc{\Funex}{\on{Fun}_{\mrm{ex}}}
\nc{\ex}{\mrm{ex}}
\nc{\Ind}{\on{Ind}}
\nc{\Pro}{\on{Pro}}
\nc{\Kar}{\on{Kar}}
\nc{\Obj}{\on{Obj}}
\nc{\Calk}{\mrm{Calk}}
\nc{\Cone}{\mrm{Cone}}
\nc{\Cof}{\mrm{cof}}
\nc{\splt}{\mrm{split}}

\nc{\scr}{\term{simplicial commutative ring}}
\nc{\scrs}{\term{simplicial commutative rings}}

\nc{\Einfring}{\term{$\Einfty$-ring}}
\nc{\Einfrings}{\term{$\Einfty$-rings}}

\nc{\Ering}{\term{$\sE_1$-ring}}
\nc{\Erings}{\term{$\sE_1$-rings}}

\nc{\inftyCat}{\term{$\infty$-category}}
\nc{\inftyCats}{\term{$\infty$-categories}}

\nc{\inftyTop}{\term{$\infty$-topos}}
\nc{\inftyTops}{\term{$\infty$-toposes}}

\nc{\inftyGrpd}{\term{$\infty$-groupoid}}
\nc{\inftyGrpds}{\term{$\infty$-groupoids}}
\title{}

\begin{document}

\title{Every motive is the motive of a stable $\infty$-category}
\author{Maxime Ramzi}
\address{Fachbereich Mathematik und Informatik\\
Universit\"at M\"unster\\
48149 M\"unster\\
Germany
}
\email{\href{mailto:mramzi@uni-muenster.de}{mramzi@uni-muenster.de}}

\author{Vladimir Sosnilo}
\address{
Fakult\"at f\"ur Mathematik\\
Universit{\"a}t Regensburg\\
93040 Regensburg\\
Germany
}
\email{\href{mailto:vsosnilo@gmail.com}{vsosnilo@gmail.com}}

\author{Christoph Winges}
\address{
Fakult\"at f\"ur Mathematik\\
Universit{\"a}t Regensburg\\
93040 Regensburg\\
Germany
}
\email{\href{mailto:christoph.winges@ur.de}{christoph.winges@ur.de}}

\bibliographystyle{alphamod}

\begin{abstract}
We define a class of motivic equivalences of small stable $\infty$-categories $W_{\mathrm{mot}}$ and show that the Dwyer--Kan 
localization functor $\mathrm{Cat}^{\mathrm{perf}}_\infty \to \mathrm{Cat}^{\mathrm{perf}}_\infty[W_{\mathrm{mot}}^{-1}]$ is 
the universal localizing invariant in the sense of Blumberg--Gepner--Tabuada. 
In particular, we show that every object in its target $\mathcal{M}_{\mathrm{loc}}$ can be represented as 
$\mathcal{U}_{\mathrm{loc}}(\mathcal{C})$ for some small stable $\infty$-category $\mathcal{C}$. As another consequence, and using work of Efimov,
we improve the universal property of $\mathcal{M}_{\mathrm{loc}}$ and show that any $\aleph_1$-finitary localizing invariant 
factors uniquely through it. 
\end{abstract}

\maketitle

\setcounter{tocdepth}{1}
\tableofcontents

\section{Introduction}

One of the many challenges in the subject of K-theory has always been to give it a general enough definition 
which would be both computationally useful and would satisfy a natural universal property. 
One can give a very general universal construction of K-theory in the context of localizing invariants 
\cite{blumberg2013universal, CisinskiTabuada}. To do so, one defines a universal localizing invariant 
\[
\Cat^\perf_\infty \to \Motnc
\]
and proves that nonconnective K-theory is corepresented by the symmetric monoidal unit in $\Motnc$. 
This definition is very satisfying from the theoretical point of view, but is quite hard to approach practically, 
because of our lack of understanding of what $\Motnc$ looks like. In particular, 
there is no concrete way to think about objects of $\Motnc$, limits and colimits of those, and the general 
mapping spaces between them. 
The goal of this paper is to give a new, more concrete construction of $\Motnc$ which is independent of the construction in \cite{blumberg2013universal}.

Let us recall some of the basic notions required to state the main theorem.

\begin{defn}\label{defn:KV_sequence}
We say that a diagram in $\Cat^{\perf}_{\infty}$
\[
\sC \stackrel{f}\to \sD \to \sE
\]
is a {\bf Karoubi sequence} if it is a cofiber sequence and $f$ is fully faithful.
\end{defn}

Since we are working in $\Cat^\perf_\infty$, the $\infty$-category of idempotent complete stable $\infty$-categories, the 
cofiber of a fully faithful functor is given by the idempotent completion of the usual Verdier quotient of stable 
$\infty$-categories. We will call this the {\bf Karoubi quotient}.

\begin{defn}\label{defn:localizing-finitary}
A functor $\Cat^{\perf}_{\infty} \to \sX$ is 
\begin{enumerate}
 \item a {\bf localizing invariant} if it sends $0$ to the final object of $\sX$ and Karoubi sequences to fibre sequences in $\sX$;
 \item {\bf finitary} if it preserves filtered colimits.
\end{enumerate}
We denote by $\Fun^{\mrm{loc}, \mrm{fin}}(\mrm{Cat}^{\perf}_{\infty}, \sX)$ the full subcategory of $\Fun(\Cat^{\perf}_{\infty}, \sX)$ consisting of finitary localizing invariants.
\end{defn}

\begin{exam}
The nonconnective K-theory functor $\K : \Cat^\perf_\infty \to \Spt$ is a finitary localizing invariant. 
\end{exam}

\begin{defn}\label{def:motivic-equivalence}
 An exact functor $f : \sC \to \sD$ of small stable $\infty$-categories is a {\bf motivic equivalence} if $L(f)$ is an 
 equivalence for every finitary localizing invariant $L : \Cat^\perf_\infty \to \sX$ with $\sX$ a stable  
 $\infty$-category. 
\end{defn}

\begin{thm}\label{thm:main-intro}
 Let $W_\mrm{mot}$ denote the class of motivic equivalences in $\Cat^\perf_\infty$.
 \begin{enumerate}
  \item The Dwyer--Kan localization $\Motnc := \Cat^\perf_\infty[W^{-1}_\mrm{mot}]$ is a locally small and stable presentable $\infty$-category.
  \item The symmetric monoidal structure on $\Cat^\perf_\infty$ induces a symmetric monoidal structure on $\Motnc$ such that the localization functor $\gamma : \Cat^\perf_\infty \to \Motnc$ refines to a symmetric monoidal functor.
  \item The localization functor $\gamma$ is the universal finitary localizing invariant: the functor $\gamma$ is a finitary localizing invariant, and for every stable presentable
  $\infty$-category $\sX$, restriction along $\gamma$ induces an equivalence
  \[ \gamma^* : \Fun^\mrm{L}(\Motnc,\sX) \xrightarrow{\sim} \Fun^{\mrm{loc},\mrm{fin}}(\Cat^\perf_\infty,\sX). \]
 \end{enumerate} 
\end{thm}

This gives a completely new computational and theoretical approach to studying the universal localizing invariant. In particular this result can be used to study functors out of $\Motnc$ that are not exact, and so do not correspond to localizing invariants, in relation to e.g. multiplicative structures or polynomial functoriality. 

One of the key observations (see Proposition~\ref{prop:cat-perf-cofibration-cat}) lies in the fact that the triple $(\Cat^\perf_\infty, W_{\mrm{mot}}, \Cof)$ gives $\Cat^\perf_\infty$ the structure of a {\it cofibration category} in the sense of Cisinski \cite[Chapter~7]{cisinski-rl}, where $\Cof$ is the class of fully faithful functors.
This is reminiscent of classical homotopy theory, where one can approach the $\infty$-category of spaces through the model category of simplicial sets.
In this sense, one can think of $\Motnc$ as a {\it homotopy theory} of stable $\infty$-categories. 

To prove stability of the localization $\Cat^\perf_\infty[W^{-1}_\mrm{mot}]$, we rely on the fact that desuspensions in the category of motives can be categorified in terms of the Grayson construction.
In principle, the necessary statements can be deduced from \cite{KasprowskiWinges}, but we include a discussion of the Grayson construction for the reader's convenience in Appendix~\ref{sec:reminder_on_the_grayson_construction}.
While doing so, we also seize the opportunity to provide a more geometric interpretation of the Grayson construction, and remove an unnecessary assumption from \cite{KasprowskiWinges}.

Appendix~\ref{app:linear} sketches a generalization of our main results to the setting of $\sV$-linear categories, where $\sV$ is any presentably symmetric monoidal stable $\infty$-category.

%
%

We demonstrate the utility of our results via several applications. 

\sssec{Every spectrum is the K-theory of a stable category}
An immediate corollary of Theorem~\ref{thm:main-intro} is that every object of $\Motnc$ is equivalent to 
$\gamma(\sC)$ for some $\sC \in \Cat^\perf_\infty$. 
In our previous paper \cite{rsw:every-spectrum-is-k-theory} we proved the following result. 

\begin{thm}\label{thm:main-spectra}
 For every spectrum $M$, there exists a small stable $\infty$-category $\sC$ with $\K(\sC) \simeq M$.
\end{thm}

We obtain an alternative proof of Theorem~\ref{thm:main-spectra} 
by combining Theorem~\ref{thm:main-intro} with the following lemma.

\begin{lem}\label{lem:motivic-section}
 The functor $\Motnc \to \Spt$ induced by nonconnective algebraic K-theory admits a section.
\end{lem}
\begin{proof}
 Since K-theory is lax symmetric monoidal, the lax symmetric monoidal and colimit-preserving functor $\Motnc \to \Spt$  induced by $\K$ factors over $\Mod(\K(\mbf{S}))$.
 Since the Dennis--Waldhausen trace $\K(\mbf{S}) \to \mbf{S}$ exhibits the sphere spectrum as a retract of $\K(\mbf{S})$ in the category of ring spectra, the forgetful functor $\Mod(\K(\mbf{S})) \to \Mod(\mbf{S}) \simeq \Spt$ admits a section.
 
 Moreover, we have $\map_{\Motnc}(\sU_{\mrm{loc}}(\Spt^\omega),\sU_{\mrm{loc}}(\Spt^\omega)) \simeq \K(\mbf{S})$ as ring spectra, which induces a fully faithful section $\Mod(\K(\mbf{S})) \to \Motnc$ to the functor $\Motnc \to \Mod(\K(\mbf{S}))$ because $\sU_{\mrm{loc}}(\Spt^\omega)$ is compact by \cite[Theorem~9.8]{blumberg2013universal} (see also Corollary~\ref{cor:Kthcorep} below).
\end{proof}

\sssec{Improving the universal property}
Even though the class $W_{\mrm{mot}}$ in Theorem~\ref{thm:main-intro} is rather inexplicit, we are able to give a more 
tangible description of motivic invariance, at least for $\aleph_1$-finitary localizing invariants, and using currently unpublished work of Sasha Efimov.
 
More precisely, we construct a very explicit class of motivic equivalences $W_{\mrm{simple}}$ in $\Cat^\perf_\infty$ such 
that for any $\infty$-category $\sX$ the restriction functor
\[
\Fun^{\aleph_1-\mathrm{fin}}(\Motnc,\sX) \to \Fun^{\aleph_1-\mathrm{fin}}_{W_{\mrm{simple}}}(\Cat^\perf_\infty,\sX)
\]
is an equivalence. Here the superscript ``$\aleph_1$-fin'' denotes the the subcategory of functors that preserve 
$\aleph_1$-filtered colimits and the subscript ``$W_{\mrm{simple}}$'' denotes the subcategory of functors that send elements of 
$W_{\mrm{simple}}$ to equivalences.

The explicit description allows us to check directly whether a functor $\Cat^\perf_\infty \to \sX$ factors over $\gamma$. 
In particular we obtain the following stronger universal property of $\Motnc$. 

\begin{thm}[{Theorem~\ref{thm:universal_omega_1_finitary}}]\label{thm:omega_1intro}
The functor $\gamma : \Cat^\perf_\infty \to \Motnc$ is the universal $\aleph_1$-finitary localizing invariant: 
for every stable $\infty$-category $\sX$, restriction along $\gamma$ induces an equivalence
\[ \gamma^* : \Fun^{\aleph_1-\mathrm{fin}}_{\mathrm{ex}}(\Motnc,\sX) \xrightarrow{\sim} \Fun^{\mrm{loc},\aleph_1-\mathrm{fin}}(\Cat^\perf_\infty,\sX). \]
\end{thm}

Theorem~\ref{thm:omega_1intro} can be used to verify the following.

\begin{thm}[Theorem~\ref{thm:countable_products}]
 The universal localizing invariant $\sU_{\mrm{loc}}$ preserves countable products.
\end{thm}


\ssec{Notation}

\begin{itemize}
\item We will denote by $\sS$ the $\infty$-category of spaces and by $\Spt$ the $\infty$-category of spectra. 
\item We will denote by $\Cat^{\perf}_{\infty}$ the $\infty$-category of idempotent complete, small stable $\infty$-categories and exact functors.
For $\sC, \sD \in \Cat^{\perf}_{\infty}$, we denote by $\Funex(\sC,\sD)$ the $\infty$-category of exact functors from $\sC$ to $\sD$. 
\item For a cardinal $\kappa$, and a cocomplete $\infty$-category $\sC$ we denote by $\sC^\kappa$ the subcategory of 
$\kappa$-compact objects. In particular, $\sC^\omega$ is the subcategory of compact objects in the usual sense. 
\item Given a stable $\infty$-category $\sC$ and $x,y \in \sC$, we denote by $\Map_\sC(x,y)$ the mapping space between these objects, and by $\map_\sC(X,Y)$ the corresponding mapping spectrum. 
\item We use the notation $\K(-)$ for the {\bf nonconnective} K-theory spectrum.
\end{itemize}

\ssec{Acknowledgements}
We are very grateful to Sasha Efimov for helpful conversations about his work, and we thank Denis-Charles Cisinski and Fabian Hebestreit for helpful discussions related to this work. 
The first author was supported by the Danish National Research Foundation through the Copenhagen Centre for Geometry and Topology (DNRF151) while much of this research was conducted,  and is funded by the Deutsche Forschungsgemeinschaft (DFG, German Research Foundation) -- Project-ID 427320536 -- SFB 1442, as well as by Germany's Excellence Strategy EXC 2044 390685587, Mathematics Münster: Dynamics--Geometry--Structure.

The second and third author are supported by the SFB 1085 ``Higher Invariants'' funded by the Deutsche Forschungsgesellschaft (DFG). 
While parts of this research were conducted, the third author was funded by Ludwig-Maximilians-Universit\"at M\"unchen.

\section{\texorpdfstring{$\Cat^\perf_\infty$ as a category of cofibrant objects}{Catperf as a category of cofibrant objects}}

Let us recall the relevant results about cofibration categories from \cite{cisinski-rl}.
Since our main reference is written for the dual notion of categories of fibrant objects, we spell out the details for the convenience of the reader.

\begin{defn}
 Let $\sC$ be an $\infty$-category.
 A collection $S$ of morphisms in $\sC$ is {\bf closed under pushouts along a collection of maps $T$} if every span $z \xleftarrow{t} x \xrightarrow{s} y$ with $s \in S$ and $t \in T$ admits a pushout, and for any choice of pushout the map $z \to z \cup_x y$ also lies in $S$.
\end{defn}

\begin{prop}[{\cite[Theorem~7.2.8]{cisinski-rl}}]\label{prop:calculus-fractions}
 Let $\sC$ be an $\infty$-category, let $W$ be a collection of morphisms in $\sC$ and let $\ell : \sC \to \sC[W^{-1}]$ be a Dwyer--Kan localization of $\sC$ at $W$.
 Fix $y \in \sC$.
 Suppose there exists a functor $p : W(y) \to \sC$ with the following properties:
 \begin{enumerate}
  \item $W(y)$ is a small $\infty$-category with an initial object $w_0$ satisfying $p(w_0) = y$;
  \item $p$ sends every morphism of the form $w_0 \to w$ in $W(y)$ to a morphism in $W$.
 \end{enumerate}
 If the functor
 \[ M_y : \colim_{w \in W(y)} \Map_\sC(-,p(w)) : \sC^\op \to \sS \]
 inverts every morphism in $W$, then the natural transformation $\Map_\sC(-,y) \to M_y \simeq \overline{M}_y \circ \ell$ exhibits the induced functor $\overline{M}_y$ as a left Kan extension along $\ell$.
 In particular, $M_y \simeq \Hom_{\sC[W^{-1}]}(\ell(-),\ell(y))$.
\end{prop}

\begin{rem}\label{rem:zig-zags}
 Since $\pi_0 : \sS \to \Set$ preserves colimits, morphisms $\ell(x) \to \ell(y)$ in the localization are represented by pairs $(f,s)$ with $f : x \to y'$ arbitrary and $s : y \to y'$ in $W$.
 It is convenient to denote such a pair as a zig-zag $x \xrightarrow{f} y'\xleftarrow{s} y$.
 It follows that two such zig-zags $(f,s)$ and $(g,t)$ represent homotopic maps if and only if they are equivalent with respect to the equivalence relation generated by diagrams of the form
 \[\begin{tikzcd}[row sep=.5em]
  & y' \ar[dd]& \\
  x\ar[ur, "f"]\ar[dr, "g"'] & & y\ar[ul, "s"']\ar[dl, "t"] \\
  & y''&
 \end{tikzcd}\]
 As explained in \cite[Remark~7.2.10]{cisinski-rl}, the morphism in $\sC[W^{-1}]$ represented by the zig-zag $x \xrightarrow{f} y'\xleftarrow{s} y$ is explicitly given by $\ell(s)^{-1}\ell(f)$.
\end{rem}

\begin{exam}\label{ex:zig-zag-compose}
 Suppose that $W$ is a subcategory of a small $\infty$-category $\sC$.
 For any $y \in \sC$, define $W(y)$ as the full subcategory of $\sC_{y/}$ spanned by the objects represented by morphisms in $W$.
 If $W$ is closed under pushouts along arbitrary maps, \cite[Theorem~7.2.16]{cisinski-rl} shows that the functor $p : W(y) \subseteq \sC_{y/} \to \sC$ induced by the target projection satisfies the assumptions of Proposition~\ref{prop:calculus-fractions}.
 
 In this case, the composition law on the homotopy category of $\sC[W^{-1}]$ admits an easy description.
 Given two zig-zags $x \xrightarrow{f} y'\xleftarrow{s} y$ and $y \xrightarrow{g} z' \xleftarrow{t} z$, we may form the pushout
 \[\begin{tikzcd}
  y\ar[r, "g"]\ar[d, "s"'] & z'\ar[d, "s'"] \\
  y'\ar[r, "g'"] & z''
 \end{tikzcd}\]
 Then $s'$ is also in $W$, and we have
 \[ \ell(t)^{-1}\ell(g)\ell(s)^{-1} \ell(f) \simeq  \ell(s't)^{-1}\ell(g'f) \]
 in $\sC[W^{-1}]$. It follows from Remark~\ref{rem:zig-zags} that the composite of $(f,s)$ and $(g,t)$ is represented by the zig-zag $x \xrightarrow{g'f} z''\xleftarrow{s't} z$.
\end{exam}

\begin{defn}[{\cite[Definitions~7.4.6,~7.4.12~and~7.5.7]{cisinski-rl}}]
A {\bf category of cofibrant objects} $(\sC,W,\Cof)$ is an $\infty$-category $\sC$ with an initial object $\emptyset$ equipped with two wide subcategories $W$ and $\Cof$, called the {\bf weak equivalences} and the {\bf cofibrations}, such that the following holds:
\begin{enumerate}
\item all maps of the form $\emptyset \to x$ are cofibrations;
\item pushouts of cofibrations along arbitrary maps exist and are cofibrations;
\item pushouts of weak equivalences along cofibrations are weak equivalences;
\item $\sC$ {\bf admits factorisations}: any map in $\sC$ can be written as a composite of a cofibration followed by a weak equivalence;
\item the subcategory $W$ satisfies the two-out-of-three property.
\end{enumerate}
\end{defn}

The main result about categories of cofibrant objects, which has precursors in work of Weiss \cite{weiss:hammock} and Blumberg and Mandell \cite{bm:abstracthomotopy}, is the following.
Recall that a class $W$ of morphisms in an $\infty$-category has the {\bf two-out-of-six property} if for any three composable maps  $w \xrightarrow{f} x \xrightarrow{g} y \xrightarrow{h} z$ with $hg \in W$ and $gf \in W$, we also have $f,g,h,hgf \in W$.

\begin{thm}\label{thm:localize-cofibration-cat}
 Let $(\sC,W,\Cof)$ be a small category of cofibrant objects and let $\ell : \sC \to \sC[W^{-1}]$ be the Dwyer--Kan localization of $\sC$ at $W$.
 \begin{enumerate}
  \item The $\infty$-category $\sC[W^{-1}]$ admits finite colimits and $\ell$ preserves the initial object and pushouts along cofibrations.
  \item Let $f : x \to y$ be a morphism in $\sC$ such that $\ell(f)$ admits a left inverse. Then there exists a morphism $g : y \to x'$ such that $gf \in W$.
  \item We have $W = \ell^{-1}(\sC[W^{-1}]^\simeq)$ if and only if $W$ satisfies the two-out-of-six property.
  \item If $\sC$ admits $\kappa$-small coproducts and the subcategory $W$ is closed under $\kappa$-small coproducts, then $\sC[W^{-1}]$ also admits $\kappa$-small coproducts and $\ell$ preserves $\kappa$-small coproducts.
 \end{enumerate}
\end{thm}
\begin{proof}
 The first assertion is \cite[Proposition~7.5.6]{cisinski-rl}.
 The second assertion follows from \cite[Corollary~7.5.19]{cisinski-rl}, but will also be part of the proof of the third assertion.
 One direction of the third assertion is obvious since equivalences in an $\infty$-category satisfy two-out-of-six.
 
 For the non-trivial direction, we need to recall a part of the proof of \cite[Proposition~7.5.6]{cisinski-rl}.
 Denote by $\Cof_W$ the subcategory of morphisms which are both cofibrations and weak equivalences.
 Ken Brown's lemma \cite[Proposition~7.4.13]{cisinski-rl} implies that $\sC[\Cof_W^{-1}] \simeq \sC[W^{-1}]$.
 Since every morphism can be factored into a cofibration followed by a weak equivalence and $W$ satisfies the two-out-of-three property, $\Cof_W$ is closed under pushouts.
 Combining Remark~\ref{rem:zig-zags} and Example~\ref{ex:zig-zag-compose}, we can represent $\ell(f)$ by the zig-zag $x \xrightarrow{f} y \xleftarrow{\id} y$, and the inverse to $\ell(f)$ is represented by a zig-zag $y \xrightarrow{g} x' \xleftarrow{t} x$.
 Forming the pushout
 \[\begin{tikzcd}
  x\ar[r, "f"]\ar[d, "t"'] & y\ar[d, "t'"] \\
  x'\ar[r, "f'"] & y'
 \end{tikzcd}\]
 the zig-zag $y \xrightarrow{f'g} y' \xleftarrow{t'} y$ represents the identity on $\ell(y)$.
 Since the identity on $\ell(y)$ is also represented by the zig-zag $y \xrightarrow{\id} y \xleftarrow{\id} y$ and $W$ satisfies two-out-of-three, it follows that $f'g$ is a weak equivalence.
 
 By a completely analogous argument, the zig-zag $x \xrightarrow{gf} x' \xleftarrow{t}$ represents the identity on $\ell(x)$, and it follows that $gf$ is a weak equivalence---this reproves the second assertion.
 If the weak equivalences satisfy two-out-of-six, we conclude that $g$ is a weak equivalence, and therefore so are $f'$ and $f$ by two-out-of-three.
 
 For the fourth assertion, observe that $\sC$ admits $\kappa$-small coproducts if and only if the diagonal functor $\Delta : \sC \to \prod_J \sC$ admits a left adjoint $\sqcup_J$ for every $\kappa$-small set $J$.
 If $W$ is closed under $\kappa$-small coproducts, we obtain for every $\kappa$-small set $J$ an induced adjunction
 \[ C : (\prod_J \sC)\big[(\prod_J W)^{-1}\big] \rightleftarrows \sC : \Delta. \]
 By \cite[Proposition~7.7.1]{cisinski-rl}, the canonical functor $(\prod_J \sC)[\left(\prod_J W\right)^{-1}] \to \prod_J \sC[W^{-1}]$ is an equivalence, so $\sC[W^{-1}]$ admits $\kappa$-small coproducts.
 As $\ell \circ \sqcup_J \simeq C \circ \ell$, the localization functor preserves such coproducts.
\end{proof}

Our key observation is that $\Cat^\perf_\infty$ admits the structure of a category of cofibrant objects.
The proof relies on the following lemma which should be well-known, but for which we have been unable to find a reference.

\begin{lem}\label{lem:localization-pullback}
 The class of right Bousfield localizations is closed under pullbacks along arbitrary functors.
 The same is true for the class of left Bousfield localizations.
\end{lem}
\begin{proof}
 It is enough to prove one of the statements, the other follows by dualising.
 Consider a pullback square
 \[\begin{tikzcd}
  \sD \times_\sC \sE \ar[r, "q"]\ar[d, "p"'] & \sE\ar[d, "g"] \\
  \sD\ar[r, "f"] & \sC
 \end{tikzcd}\]
 of $\infty$-categories in which $f$ admits a fully faithful left adjoint $f_!$.
 The unit $\omega : \id \overset{\sim}{\Rightarrow} f f_!$ induces a natural equivalence $g \simeq ff_!g$, so the universal property of the pullback provides a functor $h : \sE \to \sD \times_\sC \sE$ and a natural equivalence $\tau : qh \overset{\sim}{\Rightarrow} \id$.
 We claim that $\tau$ is the unit of an adjunction.
 As mapping spaces in pullbacks are themselves pullbacks, we need to check for $x \in \sE$  and $y \in \sD \times_\sC \sE$ that the composite
 \[ \Map_\sD(ph(x),p(y)) \mathop{\times}\limits_{\Map_\sC(fph(x),fp(y))} \Map_\sE(qh(x),q(y)) \to \Map_\sE(qh(x),q(y)) \xrightarrow{\sim} \Map_\sE(x,q(y)) \]
 is an equivalence.
 Since the structure map $\Map_\sD(ph(x),p(y)) \to \Map_\sC(fph(x),fp(y))$ is equivalent to $\Map_\sD(f_!g(x),p(y)) \to \Map_\sC(ff_!g(x),fp(y))$, it follows from the assumptions on $f$ that this is an equivalence.
 We conclude that $h$ is a fully faithful left adjoint to $q$.
\end{proof}

\begin{prop}\label{prop:cat-perf-cofibration-cat}
Consider $\Cat^\perf_\infty$ equipped with the following classes: 
\begin{enumerate}
\item $W_{\mrm{mot}}$ is the class of {\bf motivic equivalences}, see Definition~\ref{def:motivic-equivalence};
\item $\Cof$ is the class of fully faithful exact functors. 
\end{enumerate}
Then $(\Cat^\perf_\infty,W_{\mrm{mot}}, \Cof)$ is a category of cofibrant objects in which $W_{\mrm{mot}}$ satisfies the two-out-of-six property.
Moreover, an exact functor $f : \sC \to \sD$ can be functorially written as a composition
\[ \sC \to \sD \vec{\times}_{\Ind(\sD)^{\aleph_1}} \Ind(\sC)^{\aleph_1} \to \sD, \]
of a fully faithful functor and a motivic equivalence.
\end{prop}
\begin{proof}
As any equivalence is fully faithful and a motivic equivalence, we immediately see that $W_{\mrm{mot}}$ and $\Cof$ are wide subcategories.
The exact functor $0 \to \sC$ is fully faithful, and the definition of $W_{\mrm{mot}}$ makes it clear that it satisfies two-out-of-six.
The existence of functorial factorisations of the claimed form follows from \cite[Proposition~3.7]{rsw:every-spectrum-is-k-theory}; note that the proof given there uses precisely the lax pullback construction $\sD \vec{\times}_{\Ind(\sD)^{\aleph_1}} \Ind(\sC)^{\aleph_1}$.

To compute pushouts in $\Cat^\perf_\infty$, we use that the functor $\Ind : \Cat^\perf_\infty \to \PrL_{\mrm{st},\omega}$ to the $\infty$-category of compactly generated, stable, presentable $\infty$-categories and left adjoint functors preserving compact objects is an equivalence.
Its inverse is given by the functor which takes compact objects.
Since $\PrL_{\mrm{st},\omega}$ is equivalent to the opposite of the $\infty$-category $\PrR_{\mrm{st},\omega}$ of compactly generated, stable, presentable $\infty$-categories and right adjoints which preserve filtered colimits (by passing to adjoint functors), and since limits in $\PrR_{\mrm{st},\omega}$ agree with the limits of the underlying categories, this provides a construction of colimits in $\Cat^\perf_\infty$.

Let $i : \sC \to \sD$ be fully faithful and let $f : \sC \to \sC'$ be an exact functor.
Taking ind-completions and passing to right adjoints, we obtain the pullback square
\[\begin{tikzcd}
 \Ind(\sD) \times_{\Ind(\sC)} \Ind(\sC')\ar[r, "j^*"]\ar[d, "g^*"'] & \Ind(\sD)\ar[d, "f^*"] \\
 \Ind(\sD)\ar[r, "i^*"] & \Ind(\sC)
\end{tikzcd}\]
Applying Lemma~\ref{lem:localization-pullback}, we find that $j^*$ is a right Bousfield localization.
Consequently, the structure morphism $j_!^\omega : \Ind(\sD)^\omega \to \left( \Ind(\sD) \times_{\Ind(\sC)} \Ind(\sC') \right)^\omega$ of the pushout in $\Cat^\perf_\infty$ is fully faithful.

Let $i : \sC \to \sD$ be fully faithful and let $f : \sC \to \sC'$ be a motivic equivalence.
In the pushout square
\[\begin{tikzcd}
 \sC\ar[r, "i"]\ar[d, "f"'] & \sD\ar[d, "g"] \\
 \sC'\ar[r, "j"] & \sD'
\end{tikzcd}\]
the induced functor $j$ is fully faithful, as we have just shown.
Taking horizontal cofibers and applying any localizing invariant with stable target, we obtain a map of bifibre sequences which is an equivalence on the two outer terms.
It follows that $g$ is also a motivic equivalence.
\end{proof}

\section{Exhausting \texorpdfstring{$\Cat^\perf_\infty$}{Catperf} by small categories of cofibrant objects}

Ideally, we would like to apply Theorem~\ref{thm:localize-cofibration-cat} directly to the localization of $\Cat^\perf_\infty$ at the motivic equivalences.
However, it is a priori not clear that taking this localization would produce a locally small $\infty$-category.
This section addresses this problem by exhibiting a filtration of $\Cat^\perf_\infty$ by small full subcategories which inherit the structure of a category of cofibrant objects from the ambient category.
This will allow us to localize each filtration step individually and then take a colimit over these localizations afterwards to obtain $\Cat^\perf_\infty[W_{\mrm{mot}}^{-1}]$.


\begin{defn}
A cardinal $\kappa$ is {\bf countably closed} if $\lambda^{\aleph_0} < \kappa$ for all $\lambda < \kappa$. 
\end{defn}

The next lemma shows that there exists a sufficient supply of countably closed cardinals.

\begin{lem}
For any cardinal $\kappa_0$, there exists a regular, countably closed cardinal $\kappa > \kappa_0$. 
\end{lem}
\begin{proof}
Let $\mu := \max(\kappa_0, \aleph_0)$, and let $\kappa := (2^\mu)^+$.
Since successor cardinals are regular, we only have to show that $\kappa$ is countably closed.
If $\lambda < \kappa$, then $\lambda \leq 2^\mu$ and so $\lambda^{\aleph_0} \leq 2^{\mu \cdot \aleph_0} = 2^\mu < \kappa$, so $\kappa$ is indeed a countably closed cardinal, and it is clearly bigger than $\kappa_0$.
\end{proof}

Our goal is to show that the full subcategory $\Cat^{\perf,\kappa}_\infty$ of $\kappa$-compact objects in $\Cat^\perf_\infty$ inherits the structure of a category of cofibrant objects from $\Cat^\perf_\infty$ whenever $\kappa$ is a regular, countably closed cardinal.
As a first step, we will prove a characterisation of $\kappa$-compact objects in $\Cat^\perf_\infty$ for an arbitrary regular, uncountable cardinal $\kappa$.

\begin{lem}\label{lem:compact-rings}
 \ \begin{enumerate}
  \item The homotopy groups of a compact spectrum are countable.
  \item The homotopy groups of a compact ring spectrum are countable.
  \end{enumerate}
\end{lem}
\begin{proof}
  Since the sphere spectrum generates the full subcategory of compact spectra under desuspensions and finite colimits, the first assertion follows from the fact that $\pi_*\bS$ is countable.
  
  For the second assertion, consider the collection $\sR$ of ring spectra whose homotopy groups are countable.
  The free ring spectrum on a spectrum $M$ is a countable direct sum of tensor powers of $M$ and therefore has countable homotopy groups if $M$ is compact.
  Combining this with the bar resolution, we find that $\sR$ is closed under finite coproducts.
  Finally, the Bousfield-Kan formula (\cite[Corollary~12.5]{ShahThesis} (see also \cite[Proposition~5.5.8.13]{HTT}) shows that $\sR$ is closed under finite colimits.
  Thus, it contains all the compact ring spectra.
\end{proof}

\begin{prop}\label{prop:cpctcat}
Let $\kappa$ be a regular, uncountable cardinal and let $\sC$ be an idempotent complete 
$\infty$-category.
Consider the following conditions:
\begin{enumerate}
 \item\label{prop:cpctcat-1} $\sC$ is $\kappa$-compact in $\Cat^\perf_\infty$;
 \item\label{prop:cpctcat-2} $\sC$ is $\kappa$-compact in $\Cat_\infty$;
 \item\label{prop:cpctcat-4} $\sC$ has a $\kappa$-small set of equivalence classes of objects and the homotopy groups of the mapping space $\Map_\sC(x,y)$ are $\kappa$-small for every pair of objects $x,y$ in $\sC$.
\end{enumerate}

The conditions \ref{prop:cpctcat-2} and \ref{prop:cpctcat-4} are equivalent. 

If $\sC$ is furthermore stable, then all three conditions are equivalent. 
\end{prop}
\begin{proof}
 Conditions~\ref{prop:cpctcat-2} 
 and \ref{prop:cpctcat-4} are equivalent by \cite[Proposition~5.4.1.2]{HTT}. 
 
 Suppose now that $\sC$ is stable and $\kappa$-compact as an object of $\Cat_\infty$ and let $(\sD_i)_i$ be a 
 $\kappa$-filtered diagram in $\Cat^\perf_\infty$. 
 Since $\Map_{\Cat^\perf_\infty}(\sC,-)$ is a full subfunctor of $\Map_\Cat(\sC,-)$ and  the forgetful functor $\Cat^\perf_\infty \to \Cat$ preserves filtered colimits by \cite[Proposition~1.1.4.6]{HA} and \cite[Corollary~4.4.5.21]{HTT}, we find that the induced map
 \[ \colim_i \Map_{\Cat^\perf_\infty}(\sC,\sD_i) \to \Map_{\Cat^\perf_\infty}(\sC,\colim_i \sD_i) \]
 is an inclusion of path components.

 Since $\sC$ is $\kappa$-small, the collection of pushout squares in $\sC$ is also $\kappa$-small. It follows that a functor $\sC \to \colim_i \sD_i$ (which, as a functor, factors through some $\sD_{i_0}$) is exact if and only if there exists some $i$ with a map from $i_0$ such that $\sC \to \sD_i$ is exact.
 Hence $\sC$ is $\kappa$-compact in $\Cat^\perf_\infty$, showing that \ref{prop:cpctcat-2} implies \ref{prop:cpctcat-1}.
 
 We are left with showing that \ref{prop:cpctcat-1} implies \ref{prop:cpctcat-4}.
 To do so, we will first prove this implication even holds in the case $\kappa = \aleph_0$.
 So consider a compact object $\sD$ in $\Cat^\perf_\infty$.
 Writing $\sD$ as a filtered colimit of its finitely generated thick subcategories, we see that $\sD$ is finitely generated as a thick subcategory.
 By summing over a finite set of generators, we find a single generator $y$ of $\sD$.
 Consequently, $\Ind(\sD)$ is equivalent to $\Mod(\End_\sD(y))$ by the Schwede--Shipley theorem \cite[Theorem~7.1.2.1]{HA}, so $\sD \simeq \mbf{Perf}(\End_\sD(y))$.
 By \cite[Proposition~7.1.2.6]{HA}, compactness of $\sD$ forces $\End_\sD(y)$ to be a compact ring spectrum.
 Applying Lemma~\ref{lem:compact-rings}, we conclude that \ref{prop:cpctcat-4} (with $\kappa = \aleph_0$) holds for $\sD$.
 
 Let $\sC$ now be a $\kappa$-compact object in $\Cat^\perf_\infty$ with $\kappa$ regular and uncountable.
 Then $\sC$ is a $\kappa$-small filtered colimit of compact objects, all of which have countably many objects up to equivalence and mapping spaces with countable homotopy groups.
 It follows that $\sC$ has less than $\kappa$ many objects up to equivalence, and that the homotopy groups of mapping spaces in $\sC$ are $\kappa$-small.
\end{proof}
\begin{rem}
Note that for (not-necessarily stable) $\infty$-categories,  \cite[Proposition~5.4.1.2]{HTT} shows that these notions of $\kappa$-smallness agree with ``all possible definitions'' (when $\kappa$ is uncountable). 
\end{rem}

\begin{cor}\label{cor:lax-pullback-compact}
Let $\kappa$ be a regular, uncountable cardinal, and $\sC \to \sD\leftarrow \sE$ be a span of $\kappa$-compact objects in $\Cat^\perf_\infty$.
Then the lax pullback $\sC \vec{\times}_\sD \sE$ is $\kappa$-compact. 
\end{cor}
\begin{proof}
This is immediate using Proposition~\ref{prop:cpctcat}.
\end{proof}

\begin{cor}\label{cor:quotient-compact}
 Let $\kappa$ be a regular, uncountable cardinal.
 Then the $\kappa$-compact objects in $\Cat^\perf_\infty$ are closed under taking full subcategories and under taking Karoubi quotients.
\end{cor}
\begin{proof}
 For subcategories, this follows immediately from Proposition~\ref{prop:cpctcat}.
 In particular, we conclude from Corollary~\ref{cor:lax-pullback-compact} that pullbacks of $\kappa$-compact objects are $\kappa$-compact.
 
 Let $\sC$ be a full stable subcategory of $\sD \in (\Cat^\perf_\infty)^\kappa$. In particular, by the case of subcategories, $\sC$ is itself $\kappa$-compact. 
 The Karoubi quotient $\sD/\sC$ is then a finite colimit of $\kappa$-compacts, and is therefore itself $\kappa$-compact. 
%
\end{proof}

\begin{lem}\label{lem:goodisgood}
Let $\kappa$ be a regular, countably closed cardinal and let $\sC$ be $\kappa$-compact in $\Cat^\perf_\infty$.
Then $\Ind(\sC)^{\aleph_1}$ is also $\kappa$-compact.
\end{lem}
\begin{proof}
 We again apply Proposition~\ref{prop:cpctcat}.
 The $\aleph_1$-compact objects are generated under finite colimits by countable direct sums of compact objects, and thus the mapping spectra in $\Ind(\sC)^{\aleph_1}$ are in the thick subcategory of $\Spt$ generated by $\prod_{n\in\mathbb N} \bigoplus_{m\in\mathbb N} X_{n,m}$, where $X_{n,m}$ is a mapping spectrum in $\sC$.
 Since these have $\kappa$-small homotopy groups and $\kappa$ is countably closed, these products also have $\kappa$-small homotopy groups. 
 
 By assumption, $\sC$ has less than $\kappa$ many objects up to equivalence, so there are less than $\kappa$ many equivalence classes of countable coproducts of objects of $\sC$ because $\kappa$ is countably closed.
 Since each of the mapping spectra between such objects has $\kappa$-small homotopy groups, it follows that the stable subcategory they generate is also $\kappa$-small. 
\end{proof}

\begin{cor}\label{cor:catperfk-cofibration-cat}
Let $\kappa$ be a regular, countably closed cardinal.
Restricting both the cofibrations and the weak equivalences from $\Cat^\perf_\infty$ to the full subcategory $\Cat^{\perf,\kappa}_\infty$ yields a category of cofibrant objects $(\Cat^{\perf,\kappa}_\infty,W^\kappa_{\mrm{mot}},\Cof^\kappa)$ with functorial factorisations.

In particular, the object
\[ \Cone(f) := \left( \sD \vec{\times}_{\Ind(\sD)^{\aleph_1}} \Ind(\sC)^{\aleph_1} \right) / \sC \]
is $\kappa$-compact for every morphism $f : \sC \to \sD$ in $\Cat^{\perf,\kappa}_\infty$.
\end{cor}
\begin{proof}
 The full subcategory $\Cat^{\perf,\kappa}_\infty$ clearly contains the initial object and is closed under finite colimits.
 In light of Proposition~\ref{prop:cat-perf-cofibration-cat}, we only have to show that the lax pullback $\sD \vec{\times}_{\Ind(\sD)^{\aleph_1}} \Ind(\sC)^{\aleph_1}$ is $\kappa$-compact for every exact functor $f : \sC \to \sD$ between $\kappa$-compact objects in $\Cat^\perf_\infty$.
 This follows by combining Lemmas~\ref{lem:goodisgood} and \ref{cor:lax-pullback-compact}.
\end{proof}

\section{Motives as a localization}\label{section:MotLoc}
In this section, we prove our main theorem, namely that the Dwyer--Kan localization of the category $\Cat^\perf_\infty$ at the motivic equivalences is the universal finitary localizing invariant.
Let $\kappa$ be a regular, uncountable cardinal.
Since $\Cat^{\perf,\kappa}_\infty$ is small, the Dwyer--Kan localization
\[ \gamma_\kappa : \Cat^{\perf,\kappa}_\infty \to \sM_\kappa := \Cat^{\perf,\kappa}_\infty[(W^\kappa_{\mrm{mot}})^{-1}] \]
is also a small category.

The following observation will be useful on various occasions.

\begin{lem}\label{lem:equiv-tensor-stable}
 Let $f : \sC \to \sD$ be a motivic equivalence in $\Cat^\perf_\infty$ and let $\sE \in \Cat^\perf_\infty$ be arbitrary.
 Then $f \otimes \sE : \sC \otimes \sE \to \sD \otimes \sE$ is also a motivic equivalence.
\end{lem}
\begin{proof}
 Note that the functor $- \otimes \sE$ preserves both fully faithful functors and colimits, so it preserves Karoubi sequences and filtered colimits.
 Hence $L(- \otimes \sE)$ is a finitary localizing invariant whenever $L$ is one, which implies that $f \otimes \sE$ is a motivic equivalence as desired.
\end{proof}

\begin{prop}\label{prop:colimloc}
 Let $\kappa$ be a regular and countably closed cardinal.
  \ \begin{enumerate}
   \item $\sM_\kappa$ is pointed and admits $\kappa$-small colimits;
   \item the localization functor $\gamma_\kappa$ preserves pushouts along fully faithful functors as well as $\kappa$-small coproducts;
   \item the symmetric monoidal structure on $\Cat^{\perf,\kappa}_\infty$ induces a symmetric monoidal structure on $\sM_\kappa$ such that $\gamma_\kappa$ refines to a symmetric monoidal functor.
  \end{enumerate}
\end{prop}
\begin{proof}
  The first two assertions are immediate from Corollary~\ref{cor:catperfk-cofibration-cat} and Theorem~\ref{thm:localize-cofibration-cat}.
  For the third assertion, observe that $\Cat^{\perf,\kappa}_\infty$ is a symmetric monoidal subcategory of $\Cat^\perf_\infty$:
  the tensor unit $\Spt^\omega$ is a compact object, so it is also $\kappa$-compact.
 Let $\sC_1$ and $\sC_2$ be $\kappa$-compact, and let $\sD : J \to \Cat^\perf_\infty$ be a $\kappa$-filtered diagram.
 From the universal property of the tensor product, it follows that
 \begin{align*}
  \Funex(\sC_1 \otimes \sC_2, \colim_J \sD)
  &\simeq \Funex(\sC_1, \Funex(\sC_2,\colim_J \sD)) \\
  &\simeq \Funex(\sC_1, \colim_J \Funex(\sC_2, \sD)) \\
  &\simeq \colim_J \Funex(\sC_1, \Funex(\sC_2, \sD)) \\
  &\simeq \colim_J \Funex(\sC_1 \otimes \sC_2, \sD)
 \end{align*}
 so $\sC_1 \otimes \sC_2$ is also $\kappa$-compact.
 For the third assertion, \cite[Propsition~A.5]{nikolaus-scholze} implies that it suffices to show that $W^\kappa_\mrm{mot}$ is closed under tensor products.
 This follows readily from Lemma~\ref{lem:equiv-tensor-stable}.
\end{proof}


\begin{cor}\label{cor:finitary}
Let $\kappa$ be a regular and countably closed cardinal.
Then the localization functor $\gamma_\kappa$ preserves $\kappa$-small filtered colimits. 
\end{cor}
\begin{proof}
Let $\sC : J \to \Cat^{\perf,\kappa}_\infty$ be a $\kappa$-small filtered diagram,
and consider the comparison map $\colim_J \gamma_\kappa(\sC)\to \gamma_\kappa(\colim_J \sC)$.
Since $\sM_\kappa$ is a Dwyer--Kan localization, there exists some stable $\infty$-category $\sD$ such that $\gamma_\kappa(\sD) \simeq \colim_J \gamma_\kappa(\sC)$.
Then the comparison map can be represented by a zig-zag $\sD \xrightarrow{f} \sE \xleftarrow{\sim} \colim_J \sC$, where the second arrow is a motivic equivalence.
Since we only consider finitary localizing invariants, $f$ is a motivic equivalence, and it follows that the comparison map is an equivalence.
\end{proof}

\begin{lem}\label{lem:finite_colims}
Let $\kappa$ be a regular and countably closed cardinal. Let 
\[
\begin{tikzcd}
a \arrow[r]\arrow[d] & b\arrow[d]\\
c \arrow[r] & d
\end{tikzcd}
\]
be a diagram in $\sM_\kappa$ that becomes a pushout diagram upon applying any finitary localizing invariant. 
Then it is a pushout diagram in $\sM_\kappa$. 
\end{lem}
\begin{proof}
Since all pushout squares can be modelled by pushouts along fully faithful functors in $\Cat^\perf_\infty$, any localizing 
invariant preserves pushouts. 
Hence the map $E(b \coprod_a c) \to E(d)$ is an equivalence for any finitary localizing invariant $E$. 
In particular, presenting $b \coprod_a c \to d$ as a zig-zag in $\Cat^\perf$, we obtain that it is an equivalence in 
$\sM_\kappa$.
\end{proof}

In the next step, we prove that $\sM_\kappa$ is stable.

\begin{defn}
 Set
 \[ \Calk := \Cone(\Spt^\omega \to 0) \simeq \Ind(\Spt^\omega)^{\aleph_1} / \Spt^\omega. \]
\end{defn}

\begin{lem}\label{lem:calk-suspension}
 Let $\kappa$ be a regular and countably closed cardinal.
 The category $\Calk$ is $\kappa$-compact and the functor $- \otimes \gamma_\kappa(\Calk) : \sM_\kappa \to \sM_\kappa$ is equivalent to the suspension functor.
\end{lem}
\begin{proof}
 Since $\Spt^\omega$ is compact in $\Cat^{\perf,\kappa}_\infty$, Corollary~\ref{cor:catperfk-cofibration-cat} implies that $\Calk$ is $\kappa$-compact.
 
 Note that $- \otimes \gamma_\kappa(\Calk)$ is the functor induced by the endofunctor $- \otimes \Calk : \Cat^\perf_\infty \to \Cat^\perf_\infty$ because $\gamma_\kappa$ is symmetric monoidal.
 Tensoring with a fixed category preserves fully faithful functors and cofiber sequences, so we have a Karoubi sequence $\sC \to \sC \otimes \Ind(\sC)^{\aleph_1} \to \sC \otimes \Calk$ in $\Cat^{\perf,\kappa}_\infty$ for every $\sC$. 
 Since the middle term admits an Eilenberg swindle, the lemma follows from Proposition~\ref{prop:colimloc}.
\end{proof}

\subsubsection{} 
Our goal is to exhibit an inverse of the suspension functor in $\sM_\kappa$.
Such an inverse is provided by the Grayson construction $\sC \mapsto \Gamma\sC$ which we discuss in detail in 
Appendix~\ref{sec:reminder_on_the_grayson_construction}. 
For now we just list three important properties of it (see Construction~\ref{ssec:construction_of_gamma}, Theorem~\ref{thm:Gamma_loop}, and Lemma~\ref{lem:kappacompact_gamma}).
\begin{enumerate}\label{enum:gamma}
\item $\Gamma\sC$ is given by tensoring with a stable $\infty$-category $\Gamma$.
\item There is a sequence of stable $\infty$-categories
\[
\Gamma \to A \to I
\]
that becomes a cofiber sequence after applying any localizing invariant valued in a stable $\infty$-category. 
Moreover, there are motivic equivalences $0 \to A$ and $\Spt \to I$. 
\item $\Gamma$, $A$ and $I$ are $\kappa$-compact in $\Cat^\perf_\infty$ for any countably closed $\kappa$. 
\end{enumerate}

\begin{prop}\label{prop:mk-stable}
 Let $\kappa$ be a regular and countably closed cardinal.
 Then the functor $- \otimes \gamma_\kappa(\Gamma)$ is inverse to the suspension functor on $\sM_\kappa$.
 In particular, the $\infty$-category $\sM_\kappa$ is stable.
\end{prop}
\begin{proof}
Lemma~\ref{lem:calk-suspension} shows that $\Sigma \simeq - \otimes \gamma_\kappa(\Calk)$, so it suffices to show that the object $\gamma_\kappa(\Calk)$ is $\otimes$-invertible. Since for the same reason we have an equivalence 
\[
\gamma_\kappa(\Gamma) \otimes \gamma_\kappa(\Calk) \simeq \Sigma \gamma_\kappa(\Gamma)
\]
it suffices to show
\[
\Sigma \gamma_\kappa(\Gamma) \simeq \gamma_\kappa(\Spt^\omega).
\]
Property~(2) implies that
\[
\gamma_\kappa(\Gamma) \stackrel{\iota}\to \gamma_\kappa(A) \stackrel{\pi}\to \gamma_\kappa(I)
\]
by Lemma~\ref{lem:finite_colims}.
Using the other half of property~(2) we obtain that this cofiber sequence is equivalent to 
\[
\gamma_\kappa(\Gamma) \to 0 \to \gamma_\kappa(\Spt^\omega)
\]
and so $\Sigma \gamma_\kappa(\Gamma) \simeq \gamma_\kappa(\Spt^\omega)$. 
Hence, $\Sigma$ is invertible and $\sM_\kappa$ is stable.
\end{proof}

The next lemma will allow us to prove that $\gamma_\kappa$ gives rise to the universal finitary localizing invariant.

\begin{lem}\label{lem:kappa-localizing}
 Let $L : \Cat^\perf_\infty \to \sX$ be a functor with stable and cocomplete target $\sX$ and let $\kappa$ be a regular, uncountable cardinal.
 Then the following are equivalent:
 \begin{enumerate}
  \item $L$ is a finitary localizing invariant;
  \item the restriction of $L$ to $\Cat^{\perf,\kappa}_\infty$ sends Karoubi sequences to fiber sequences, it preserves $\kappa$-small filtered colimits, and $L$ is the $\kappa$-filtered colimit-preserving extension of $L|_{\Cat^{\perf,\kappa}_\infty}$ to $\Cat^\perf_\infty$.
 \end{enumerate}
\end{lem}
\begin{proof}
 We begin with a preliminary observation.
 Consider a filtered diagram $\sC : J \to \Cat^\perf_\infty$, and assume without loss of generality that $J$ is a poset.
 Then $J$ is the $\kappa$-filtered colimit of its $\kappa$-small filtered subsets, 
so there exists an equivalence $\colim_J \sC \simeq \colim_{F \in \sF(J)} \colim_F \sC|_F$, where $\sF(J)$ is the poset of $\kappa$-small filtered subposets of $J$. 

 Denote now by $L_\kappa$ the restriction of $L$ to $\Cat^{\perf,\kappa}_\infty$, and let
 \[ \widehat L_\kappa : \Cat^\perf_\infty \to \sX \]
 be the $\kappa$-filtered colimit-preserving extension of $L_\kappa$.
 
 If $L$ is a finitary localizing invariant, then $L_\kappa$ sends Karoubi sequences to fibre sequences and preserves $\kappa$-small filtered colimits.
 The preliminary observation implies that the canonical functor $\widehat L_\kappa \to L$ is an equivalence.
 
 We now prove the converse.
 For a Karoubi sequence $\sC \to \sD \to \sE$ in $\Cat^\perf_\infty$, write $\sD \simeq\colim_{j \in J} \sD_j$ as a $\kappa$-filtered colimit of its $\kappa$-compact thick subcategories.
Then $\sC_j := \sD_j \times_\sD \sC$ is a full subcategory of $\sD_j$.
Denoting by $\sE_j := \sD_j/\sC_j$ the Karoubi quotient, we find that the original Karoubi sequence is the $\kappa$-filtered colimit of the Karoubi sequences $\sC_j \to \sD_j \to \sE_j$.
 Corollary~\ref{cor:quotient-compact} shows that each of these sequences lies in $\Cat^{\perf,\kappa}_\infty$.
Since $L_\kappa$ sends Karoubi sequences in $\Cat^{\perf,\kappa}_\infty$ to cofiber sequences in $\sM_\kappa$, Proposition~\ref{prop:colimloc} implies that $L$ is localizing.
 Moreover, the preliminary observation implies that $L$ preserves filtered colimits because $L_\kappa$ preserves $\kappa$-small filtered colimits.
\end{proof}

We are now able to conclude in the small case:

\begin{thm}\label{thm:mainsmall}
Let $\kappa$ be a regular, countably closed cardinal. 
Then the functor
\[ \lambda_\kappa : \Cat^\perf_\infty \simeq \Ind_\kappa(\Cat^{\perf,\kappa}_\infty) \xrightarrow{\Ind_\kappa(\gamma_\kappa)} \Ind_\kappa(\sM_\kappa) \]
induced by the localization functor $\gamma_\kappa : \Cat^{\perf,\kappa}_\infty \to \sM_\kappa$ is the universal finitary localizing invariant.
\end{thm}
\begin{proof}
 The $\infty$-category $\sM_\kappa$ is stable and $\kappa$-cocomplete by Propositions~\ref{prop:colimloc} and \ref{prop:mk-stable}, so $\Ind_\kappa(\sM_\kappa)$ is a stable and presentable $\infty$-category.
 Since $\gamma_\kappa$ is localizing and preserves $\kappa$-small filtered colimits, Lemma~\ref{lem:kappa-localizing} implies that $\lambda_\kappa$ is a finitary localizing invariant.
 Consider the commutative diagram
 \[\begin{tikzcd}
  \Fun^\mrm{L}(\Ind_\kappa(\sM_\kappa),\sX)\ar[r, "\lambda_\kappa^*"]\ar[d] & \Fun^{\mrm{loc},\mrm{fin}}(\Cat^\perf_\infty,\sX)\ar[d] \\
  \Fun^{\kappa\mrm{-L}}(\sM_\kappa,\sX)\ar[r, "\gamma_\kappa^*"] & \Fun^{\mrm{loc},\mrm{fin}}(\Cat^{\perf,\kappa}_\infty,\sX)
 \end{tikzcd}\] 
 in which $\Fun^{\kappa\mrm{-L}}$ denotes the category of functors preserving $\kappa$-small colimits and the category on the bottom right denotes the category of functors sending Karoubi sequences to fiber sequences and preserving $\kappa$-small filtered colimits.
 The left vertical arrow is an equivalence by the universal property of $\Ind_\kappa$, and the right vertical arrow is an equivalence by the universal property of $\Ind_\kappa$ combined with Lemma~\ref{lem:kappa-localizing}.
 
 The bottom horizontal arrow is fully faithful since $\gamma_\kappa$ is a localization functor.
 By the definition of motivic equivalences, any finitary localizing invariant $L : \Cat^{\perf,\kappa}_\infty \to \sX$ factors over $\gamma_\kappa$.
 Since $\gamma_\kappa$ is essentially surjective, preserves $\kappa$-small coproducts and pushouts along fully faithful functors, and every pushout in $\sM_\kappa$ can be realised by a pushout along a fully faithful functor, the induced functor $\sM_\kappa \to \sX$ preserves $\kappa$-small colimits.
 Therefore, the bottom horizontal arrow is an equivalence.
\end{proof}

\begin{cor}\label{cor:kcptcountablyclosed}
 Let $\kappa$ be a regular, countably closed cardinal.
 The universal localizing invariant preserves $\kappa$-compact objects.
\end{cor}
\begin{proof}
 This is immediate from Theorem~\ref{thm:mainsmall}.
\end{proof}

\begin{cor}\label{cor:inclusion-fff}
Let $\kappa < \lambda$ be two regular, countably closed cardinals.
Then the induced functor $\sM_\kappa \to \sM_\lambda$ is fully faithful.
\end{cor}
\begin{proof}
By Theorem~\ref{thm:mainsmall}, the induced functor $\Ind_\kappa(\sM_\kappa) \to \Ind_\lambda(\sM_\lambda)$ is an equivalence.
\end{proof}

Finally, we obtain our main theorem.

\begin{proof}[Proof of Theorem~\ref{thm:main-intro}]
 Let $K$ denote the (large) linear order of regular, countably closed cardinals.
 Then Corollary~\ref{cor:inclusion-fff} implies that the functor $\colim_{\kappa \in K} \sM_\kappa \to \colim_{\kappa \in K} \Ind_\kappa(\sM_\kappa)$ is fully faithful, and it is also essentially surjective since every object of $\Cat^\perf_\infty$ is $\kappa$-compact for some $\kappa$.
 Observing that the latter colimit system is constant, Theorem~\ref{thm:mainsmall} implies that $\colim_{\kappa \in K} \sM_\kappa$ is locally small, stable and presentable.
 
 Since each $\sM_\kappa$ is the localization of $\Cat^{\perf,\kappa}_\infty$ at the class of motivic equivalences, it follows that $\colim_{\kappa \in K} \sM_\kappa$ models the localization $\Motnc := \Cat^{\perf}_\infty[W^{-1}_\mrm{mot}]$.
 In particular, $\gamma : \Cat^\perf_\infty \to \Motnc$ is the universal localizing invariant.
  
 As before, Lemma~\ref{lem:equiv-tensor-stable} together with \cite[Propsition~A.5]{nikolaus-scholze} shows that $\Motnc$ carries a symmetric monoidal structure and that $\gamma$ refines to a symmetric monoidal functor.
\end{proof}

\section{\texorpdfstring{$\aleph_1$-compact categories}{Aleph1-compact categories}}\label{sec:aleph1case}

In the previous section we have already shown that for a countably closed regular cardinal 
$\Ind_\kappa(\sM_{\kappa}) \simeq \Motnc$. 
Our goal in this section is to prove the following more precise result about the relation of $\sM_{\kappa}$ and $\Motnc$, also when $\kappa$ is not countably closed:

\begin{prop}\label{prop:bootstrapdown}
For any  regular uncountable cardinal $\kappa$, the functor $\sM_\kappa\to \Motnc$ induces an equivalence $\sM_\kappa\to \Motnc^\kappa$, and $\Motnc$ is $\kappa$-compactly generated.
\end{prop}
\begin{rem}
The proofs of this section could be somewhat simplified by using the construction of $\Motnc$ from \cite{blumberg2013universal}. To have an independent treatment of $\Motnc$, we decided to include the slightly more complicated arguments. 
\end{rem}
We need a couple of preliminaries. 
\begin{lem}\label{lem:kappa_equiv}
Let $u:\sC \to \sD$ be a filtered colimit-preserving functor between $\kappa$-compactly generated categories admitting filtered colimits, and suppose $u$ preserves $\kappa$-compact objects.
Then any $u$-equivalence $f:x\to y$ in $\sC$ is a $\kappa$-filtered colimit of $u$-equivalences between $\kappa$-compacts. 
\end{lem}
\begin{proof}
Consider the functors $u^{\Delta^1}: \sC^{\Delta^1}\to \sD^{\Delta^1}$ and the diagonal functor $\sD\to \sD^{\Delta^1}$. They both preserve $\kappa$-compacts (since they are computed pointwise in these categories, and $u$ preserves $\kappa$-compacts by assumption) and filtered colimits (for the same reason). Since $\kappa$ is uncountable, we can apply 
\cite[Proposition~2.31]{ramzi2024dualizablepresentableinftycategories}
to conclude that the pullback $\sC^{\Delta^1}\times_{\sD^{\Delta^1}}\sD$ is $\kappa$-compactly generated, and that the projection functors preserve $\kappa$-compacts. This pullback is exactly the category of $u$-equivalences, so we deduce that $f$ is a $\kappa$-filtered colimit of $u$-equivalences between $\kappa$-compacts. 
\end{proof}

\begin{lem}\label{lem:colimits_of_small_equivalences}
Let $\kappa$ be a regular uncountable cardinal.
\begin{enumerate}
 \item Any  motivic equivalence $f: \sC\to \sD$ in $\Cat^\perf_\infty$ is a $\kappa$-filtered colimit of motivic equivalences between $\kappa$-compact stable $\infty$-categories.
 \item $\Motnc$ is $\kappa$-compactly generated.
\end{enumerate}
\end{lem}
\begin{proof}
Let $\mathcal N_\kappa\subseteq \Motnc$ denote the stable full subcategory of $\Motnc$ generated under $\kappa$-small colimits by motives of $\kappa$-compact categories. The functor $\gamma: \Cat^\perf_\infty\to \Motnc$, being finitary, factors as 
$\Cat^\perf_\infty=  \Ind_\kappa((\Cat^\perf_\infty)^\kappa) \xrightarrow{\Ind_\kappa(\gamma)}\Ind_\kappa(\mathcal N_\kappa) \to \Motnc$. By Lemma \ref{lem:kappa-localizing}, $F_\kappa:= \Ind_\kappa(\gamma): \Cat^\perf_\infty\to \Ind_\kappa(\mathcal N_\kappa)$ is a finitary localizing invariant. 

It follows that $f$ is an $F_\kappa$-equivalence (since it is a motivic equivalence). Conversely, since $F_\kappa$ factors $\gamma$, any $F_\kappa$-equivalence is a motivic equivalence. Thus it suffices to prove the assertion~(1) for
$F_\kappa$-equivalences, which is a special case of Lemma~\ref{lem:kappa_equiv} because $F_\kappa$ is finitary and by 
design preserves $\kappa$-compacts.

By the universal property of $\Motnc$, the localization functor $\gamma$ also factors through $F_\kappa$, and again because of that same universal property, 
the composite $\Motnc\to \Ind_\kappa(\mathcal N_\kappa)\to \Motnc$ is the identity.
So $\Motnc$ is a retract of a $\kappa$-compactly generated category. Since $\kappa$ is uncountable, we can apply \cite[Corollary~2.33]{ramzi2024dualizablepresentableinftycategories} to conclude that $\Motnc$ is $\kappa$-compactly generated. 
\end{proof}

We will then need the following technical statement:
\begin{lem}\label{lem:filtcolimiff}
Let $\kappa$ be an arbitrary regular cardinal. A functor $\overline{f}:\Motnc\to \mathcal{E}$ preserves $\kappa$-filtered colimits if and only if its restriction $f$ to $\Cat^\perf_\infty$ does. 
\end{lem}
\begin{proof}
We note that ``only if'' is clear, so we focus on ``if''.  This is essentially the same fact as 
\cite[Section~7.7.7]{cisinski-rl}, but this is phrased using all colimits rather than just a specified class, and the proof given there does not adapt so well to this situation, so we give a full argument.

Let $f:\Cat^\perf_\infty\to\mathcal{E}$ be a functor that sends weak equivalences to equivalences and preserves 
$\kappa$-filtered colimits, where $\mathcal E$ is some category with $\kappa$-filtered colimits. 
By \cite[Lemma~5.3.1.18]{HTT}, 
to prove that the induced functor $\overline{f}: \Motnc\to \mathcal{E}$ preserves $\kappa$-filtered colimits, it suffices to consider diagrams indexed by $\kappa$-filtered $1$-categories (in fact, posets). For these, we may apply 
\cite[Theorem~7.9.8]{cisinski-rl} to lift the diagrams to $\Cat^\perf_\infty$ (the result does apply since 
$\Cat^\perf_\infty$ has functorial factorizations and is therefore hereditary in the sense of 
\cite[Definition~7.9.4]{cisinski-rl}), and use the fact that both $f$ and the localization functor 
$\Cat^\perf_\infty\to \Motnc$ preserve $\kappa$-filtered colimits to conclude. 
\end{proof}

We can now prove a slightly weaker version of Proposition~\ref{prop:bootstrapdown}.
\begin{lem}\label{lm:bootstrapdownidem}
For any regular uncountable cardinal $\kappa$, the functor $\sM_\kappa\to \Motnc$ induces an equivalence $\Ind_\kappa(\sM_\kappa)\simeq \Motnc$.
\end{lem}
\begin{proof}
Let $\mathcal{E}$ be an arbitrary $\infty$-category with $\kappa$-filtered colimits. We wish to prove that restriction along the functor 
$\sM_\kappa\to \Motnc$ induces an equivalence 
\[
\Fun^{\kappa-\mathrm{fin}}(\Motnc,\mathcal{E})\to \Fun(\sM_\kappa,\mathcal{E}).   
\]
By definition, it is equivalent to the restriction functor 
$\Fun^{\kappa-\mathrm{fin}}(\Motnc,\mathcal{E})\to \Fun_{W^\kappa_{\mrm{mot}}}((\Cat^\perf_\infty)^\kappa,\mathcal{E})$.

Since the slice categories $(\Cat^\perf_\infty)^\kappa_{/\sC}$ are $\kappa$-filtered for every $\sC\in\Cat^\perf_\infty$, we find that any functor $(\Cat^\perf_\infty)^\kappa\to \mathcal{E}$ admits a (pointwise) left Kan extension. 

Using the universal property of left Kan extensions, together with the fact that $\Cat^\perf_\infty\to \Motnc$ 
preserves filtered colimits, we see that the above restriction functor is equivalent to the restriction
\[
\Fun^{\kappa-\mathrm{fin}}(\Motnc,\mathcal{E})\to \Fun^{\kappa-\mathrm{fin}}_{W^\kappa_{\mrm{mot}}}(\Cat^\perf_\infty,\mathcal{E})
\]
 along $\Cat^\perf_\infty\to \Motnc$.
To conclude, we are thus left with two facts: first, that $\Ind_\kappa(W^\kappa_{\mrm{mot}}) = 
W_{\mrm{mot}}$, and second, that a functor $\Motnc\to \mathcal{E}$ preserves $\kappa$-filtered colimits if and only if its 
restriction to $\Cat^\perf_\infty$ does. Both facts were proved above, see Lemmas \ref{lem:colimits_of_small_equivalences} 
and \ref{lem:filtcolimiff}. 
%
%
\end{proof}

As an immediate corollary, we obtain the following slight extension of Corollary \ref{cor:kcptcountablyclosed}:
\begin{cor}\label{cor:kcpt}
Let $\kappa$ be a regular uncountable cardinal. The universal localizing invariant preserves $\kappa$-compacts.
\end{cor}

From Lemma~\ref{lm:bootstrapdownidem}, it also follows that $\sM_\kappa$ embeds fully faithfully inside $\Motnc^\kappa$. We 
now wish to prove that this embedding is an equivalence, i.e.\ it is essentially surjective. First, we need an abstract statement 
about colimits that are $\kappa$-compact:
\begin{lem}\label{lem:colimits_compact}
Let $\sD$ be an $\infty$-category with filtered colimits, $\kappa$ a regular uncountable cardinal and $d\in \sD$ a $\kappa$-compact object. Let $J$ be a $\kappa$-filtered $\infty$-category admitting $\kappa$-small filtered colimits, and 
$d_\bullet: J\to \sD^\kappa$ a diagram preserving them, with colimit $d$. There exists $j\in J$ with $d_j\simeq d$. 
\end{lem}
\begin{proof}
Since $d$ is $\kappa$-compact and $J$ is $\kappa$-filtered, the equivalence $d\simeq \colim_J d_j$ factors through some 
$d_{j_0}$. Now the composite $d_{j_0}\to d \to d_{j_0} \to \colim_J d_j$ agrees with the canonical map, so by 
$\kappa$-compactness of $d_{j_0}$, there exists some $j_1$ with a map $j_0\to j_1$ such that the composite 
$d_{j_0}\to d\to d_{j_1}$ agrees with the canonical map $d_{j_0}\to d_{j_1}$. Inductively, we construct a diagram 
$j_\bullet: \mathbb N\to J$ such that for each $n$, $d_{j_n}\to d_{j_{n+1}}$ factors through $d$ and the composite 
$d\to d_{j_n}\to d$ is the identity of $d$. The colimit $\colim_n d_{j_n}$ is thus equivalent to $d$, and by assumption 
(since $\mathbb N$ is $\kappa$-small filtered) equivalent to $d_{\colim_n j_n}$. So we may pick $j=\colim_n j_n$.
\end{proof}
Finally, we may conclude:
\begin{proof}[Proof of Proposition~\ref{prop:bootstrapdown}]
As indicated above, we need to prove that every $\kappa$-compact motive $M$ is $\gamma(\sC)$ for some $\kappa$-compact 
stable $\infty$-category $\sC$. Since $\gamma:\Cat^\perf_\infty \to \Motnc$ is essentially surjective, 
$M\simeq\gamma(\sC)$ for some $\sC$, not necessarily $\kappa$-compact. 
By writing $\sC\simeq \colim_J \sC_j$ as a canonical $\kappa$-filtered colimit of $\kappa$-compact stable $\infty$-categories, that is, with $J=(\Cat^\perf_\infty)_{/\sC}$, and 
using that $\gamma$ preserves $\kappa$-compactness by Corollary \ref{cor:kcpt}, we may apply
Lemma~\ref{lem:colimits_compact} to  
\[
J \to \Cat^\perf_\infty \stackrel{\gamma}\to \Motnc,
\] 
so we find that $M\simeq\gamma(\sC_j)$ for some $j\in J$. 
\end{proof}

\section{The universal \texorpdfstring{$\aleph_1$}{Aleph1}-finitary invariant}

One of the main subtleties of the theory of localizing motives is the necessity to 
consider finitary invariants in the definition. 
One has to impose some finitary assumptions in order not to run into set-theoretic issues but one 
may also consider the $\kappa$-finitary version of this invariant $\mathcal{M}_{\mrm{loc},\kappa}$ for any regular 
uncountable cardinal $\kappa$ (see also Remark~\ref{rem:nonfinitary}). 
It is worth pointing out that some important localizing invariants such as $TC$, $TR$ and Selmer K-theory 
are only $\aleph_1$-finitary but not finitary. In this sense $\mathcal{M}_{\mrm{loc},\kappa}$ for $\kappa>\omega$ 
is in principle a more fundamental invariant. 

The first nontrivial case to consider is $\kappa=\aleph_1$. 
Surprisingly, it turns out that in this case one obtains the same $\infty$-category.

\begin{thm}\label{thm:universal_omega_1_finitary}
The functor $\mathcal{U}_\mrm{loc}:\Cat^\perf_\infty \to \Motnc$ is the universal $\aleph_1$-finitary localizing invariant: 
for every stable $\infty$-category $\sX$, restriction along $\gamma$ induces an equivalence
\[ 
\gamma^* : \Fun^{\aleph_1-\mathrm{fin}}_{\mathrm{ex}}(\Motnc,\sX) \xrightarrow{\sim} \Fun^{\mrm{loc},\aleph_1-\mathrm{fin}}(\Cat^\perf_\infty,\sX). 
\]
\end{thm}
The proof of this theorem occupies the entire section.

\ssec{Splitting invariants}
The functor $\Cat^\perf_\infty \to \Motnc$ factors through the {\it universal splitting invariant}
\[
\Cat^\perf_\infty \to \mathcal{M}_\splt.
\]
As an intermediate step, we prove that this functor is already very close to being a localization.


\begin{defn}\label{defn:additive}
Let $\sX$ be an $\infty$-category. 
A finite product-preserving functor $E:\Cat^\perf_\infty \to \sX$ is said to be a splitting invariant if 
every semiorthogonal decomposition
\[\begin{tikzcd} \sC\ar[r, shift left=1, "i"] & \sD\ar[l, shift left=1, "r"]\ar[r, shift left=1, "p"] & \sE\ar[l ,shift left, "j"] \end{tikzcd}\]
induces an equivalence
\[ (E(r), E(p)) : E(\sD) \xrightarrow{\sim} E(\sC) \times E(\sE). \]
\end{defn}

\begin{rem}
 Splitting invariants usually go by the name of ``additive invariants'' in the literature.
 We have chosen to use a different term to avoid confusion with additive functors, and additive presheaves specifically.
 See in particular Example~\ref{exam:universal_additive} below, which would use unnecessarily ambiguous terminology if we had to speak about ``additive presheaves''.
 
 Note that we will consequently also refer to the target of the universal splitting invariant as the $\infty$-category of splitting motives, rather than additive motives.
\end{rem}

\begin{exam}\label{exam:mapss2}
The $\infty$-category of fiber sequences $S_2\sC$ in a given small stable $\infty$-category $\sC$ admits 
a semiorthogonal decomposition into two copies of $\sC$. For a finite product-preserving functor to send this 
semiorthogonal decomposition to a product means inverting the map $s_{\sC}:S_2\sC \to \sC \times \sC$ given by taking 
the first and the last components of a fiber sequence. 
\end{exam}

\begin{lem}\label{lem:splitting-grouplike}
 Let $E : \Cat^\perf_\infty \to \sX$ be a splitting invariant.
 Then $E$ lifts uniquely to a splitting invariant $\Cat^\perf_\infty \to \mrm{CGrp}(\sX)$ valued in commutative group objects in $\sX$.
\end{lem}
\begin{proof}
 Since $E$ preserves finite products, \cite[Corollary~2.5]{ggn} shows that $E$ lifts uniquely to a functor
 \[ E : \Cat^\perf_\infty \to \mrm{CMon}(\sX). \]
 So we only have to show that $E(\sC)$ is grouplike for every $\sC \in \Cat^\perf_\infty$.
 Since $E$ inverts $s_\sC$, it follows that the maps induced by the exact functors
 \[ s : \sC \to S_2\sC,\ x \mapsto (x \to x \oplus \Sigma x \to \Sigma x),\quad\text{and}\quad t : \sC \to S_2\sC,\ x \mapsto (x \to 0 \to \Sigma x), \]
 are homotopic.
 Consequently, $0 \simeq E(\id + \Sigma)$.
 It follows that
 \[ \begin{pmatrix} \id & \Sigma \\ 0 & \id \end{pmatrix} : \sC \times \sC \to \sC \times \sC \]
 becomes an inverse of the shear map upon application of $E$.

 The of $E$ is also a splitting invariant because limits in $\mrm{CGrp}(\sX)$ are computed in $\sX$.
\end{proof}

\begin{exam}\label{exam:universal_additive}
Let $\kappa$ be a regular cardinal. We consider the initial $\kappa$-finitary splitting invariant valued in a presentable $\infty$-category:
\[
\mathcal{U}_{\splt, \kappa}:\Cat^\perf_\infty \to \mathcal{M}_{\splt, \kappa}^\mrm{un}.
\] 
Explicitly, $\mathcal{M}_\splt^\mrm{un}$ may be constructed as the localization of 
$\mrm{Fun}((\Cat^{\perf,\kappa}_\infty)^{\mrm{op}}, \Spc)$ with respect to morphisms 
\[
\yo(\sC_1) \coprod \yo(\sC_2)\to \yo(\sC)
\] 
for all $\sC \in \Cat^{\perf,\kappa}_\infty$ equipped with a semiorthogonal decomposition into $\sC_1$ and $\sC_2$ (see \cite[Section~6]{blumberg2013universal}). 
Equivalently, we may identify $\mathcal{M}_{\splt, \kappa}^\mrm{un}$ with the $\infty$-category of presheaves that are 
splitting in the sense of Definition~\ref{defn:additive}:
\[
\mathcal{M}_{\splt, \kappa}^\mrm{un} \simeq \mrm{Fun}^{\splt}((\Cat^{\perf,\kappa}_\infty)^{\mrm{op}}, \Spc).
\]
This further leads to 
the universal stable $\kappa$-finitary splitting invariant defined as
\[
\Cat^\perf_\infty \to \mrm{Stab}(\mathcal{M}_{\splt, \kappa}^\mrm{un}) =: \mathcal{M}_{\splt, \kappa}.
\]
By Lemma~\ref{lem:splitting-grouplike}, we may in fact also identify $\mathcal{M}_{\splt, \kappa}^\mrm{un}$ with the 
$\infty$-category of splitting presheaves valued in {\it connective spectra} 
\[
\mrm{Fun}^{\splt}((\Cat^{\perf,\kappa}_\infty)^{\mrm{op}}, \mrm{Spt}_{\ge 0})
\]
and $\mathcal{M}_{\splt, \kappa}$ with 
\[
\mrm{Fun}^{\splt}((\Cat^{\perf,\kappa}_\infty)^{\mrm{op}}, \mrm{Spt}).
\]
In particular, $\mathcal{M}_{\splt, \kappa}^\mrm{un}$ is a prestable $\infty$-category, and the functor 
$\mathcal{M}_{\splt, \kappa}^\mrm{un} \to \mathcal{M}_{\splt, \kappa}$ is fully faithful.
\end{exam}

\begin{exam}\label{exam:KFun}
For any idempotent complete stable $\infty$-category $\sC$, the functor $\sD \mapsto \K^{\mathrm {cn}}(\Fun_\ex(\sC, \sD))$ is a splitting invariant. 
Via the Yoneda lemma, the identity functor on $\sC$ induces a natural transformation
\[
\Maps_{\mathcal{M}_{\splt, \kappa}}(\mathcal{U}_{\splt,\kappa}(\sC),\mathcal{U}_{\splt,\kappa}(-)) \to \K^{\mathrm {cn}}(\Fun_\ex(\sC, -))
\]
for any $\sC \in (\Cat^\perf_\infty)^\kappa$. 
\end{exam}

\begin{prop}\label{prop:additive_morphisms}
For any $\sC \in (\Cat^\perf_\infty)^\kappa$ and $\sD \in \Cat^\perf_\infty$ the map 
\[
\Maps_{\mathcal{M}_{\splt, \kappa}}(\mathcal{U}_{\splt,\kappa}(\sC),\mathcal{U}_{\splt,\kappa}(\sD)) \to \K^{\mathrm {cn}}(\Fun_\ex(\sC, \sD))
\]
is an equivalence.
\end{prop}
\begin{proof}
For any $\sC$, precomposing with the adjoint functors $- \otimes \sC$ and $\Fun_\ex(\sC, -)$ induces an adjunction 
\[
\begin{tikzcd}
\Fun((\Cat^{\perf,\kappa}_\infty)^{\mathrm{op}}, \Spt) \arrow[r, shift left, "L"] &\Fun((\Cat^{\perf,\kappa}_\infty)^{\mathrm{op}}, \Spt)\arrow[l,shift left, "R"]
\end{tikzcd}
\]
that also restricts to an adjunction on $\mathcal{M}_{\splt, \kappa}$ because both $E(\sC \otimes -)$ and $E(\Fun_\ex(\sC, -))$ are splitting invariants for a splitting invariant $E$. 
In particular, in the commutative diagram
\[
\begin{tikzcd}
\Maps_{\mathcal{M}_{\splt, \kappa}}(\mathcal{U}_{\splt,\kappa}(\sC), \mathcal{U}_{\splt,\kappa}(\sD)) \arrow[r,"\simeq"] \arrow[dr] &
\Maps_{\mathcal{M}_{\splt, \kappa}}(\mathcal{U}_{\splt,\kappa}(\Spt), \mathcal{U}_{\splt,\kappa}(\Fun_\ex(\sC,\sD)))\arrow[d, "\simeq"]\\
& \K^{\mathrm{cn}}(\Fun_\ex(\sC,\sD))
\end{tikzcd}
\]
the horizontal map is an equivalence. By \cite[Theorem~4.2]{polynomial} the vertical map is also an equivalence, therefore so is the diagonal one. 
\end{proof}

The following result is a slightly modified version of \cite[Proposition~4.6]{polynomial}. 

\begin{prop}\label{prop:additive_loc}
A finite product-preserving functor $E:\Cat^\perf_\infty \to \sE$ is splitting if and only if it inverts 
the maps $s_{\sC}$ from Example~\ref{exam:mapss2} for all $\sC \in \Cat^\perf_\infty$. Moreover, for a regular cardinal $\kappa$ and a $\kappa$-finitary finite product-preserving functor both conditions are also equivalent to inverting $s_{\sC}$ for all 
$\sC \in \Cat^{\perf,\kappa}_\infty$.
\end{prop}
\begin{proof}
If $E$ is splitting, then it also inverts $s_{\sC}$ (see Example~\ref{exam:mapss2}). 
Now we assume that $E$ inverts $s_{\sC}$ for all $\sC \in \Cat^\perf_\infty$. 
Consider $\sC$ with a semiorthogonal decomposition into $\sC_1$ and $\sC_2$. 
Denote the projection maps $\sC \to \sC_i \to \sC$ by $\pi_i$. The canonical and the trivial decompositions into the components define two functors 
\[
c_1,c_2:\sC \rightrightarrows S_2\sC
\]
which become equivalent after composing with $s_{\sC}: S_2\sC \to \sC \times \sC$. 
We have $E(c_1) \simeq E(c_2)$ and so, the composites
\[
E(\sC) \rightrightarrows E(S_2\sC) \to E(\sC)
\]
which are given by the identity map and $E(\pi_1 \oplus \pi_2)$, respectively, are homotopic.

If $E$ is $\kappa$-finitary, inverting $s_{\sC}$ for all $\sC \in \Cat^{\perf,\kappa}_\infty$ automatically implies inverting 
$s_{\sC}$ for arbitrary $\sC$, since we have $s_{\sC} \simeq \colim s_{\sC_i}$ for some ind-system of $\kappa$-compact 
$\infty$-categories $\sC_i$. This shows the `moreover' part of the statement.
\end{proof}

Let $\kappa$ be a regular cardinal. Denote by $W_{s, \kappa}$ 
the set of morphisms $s_{\sC}:S_2\sC \to \sC \times \sC$ for all $\sC \in \Cat^{\perf,\kappa}_\infty$.

\begin{cor}\label{cor:motadd}
Let $\mrm{Rep}\mathcal{M}_{\splt, \kappa} \subset \mathcal{M}_{\splt, \kappa}$ denote the full subcategory of 
additive motives spanned by the ones that are of the form $\mathcal U_{\splt}(\sC)$ for some 
$\sC\in(\Cat^\perf_\infty)^\kappa$. 
Then the functor $\Cat^{\perf,\kappa}_\infty \to \mrm{Rep}\mathcal{M}_{\splt, \kappa}$ is a Dwyer--Kan localization at 
$W_{s,\kappa}$. 
\end{cor}
\begin{proof}
Consider the localization functor 
\[
U:\Cat^{\perf,\kappa}_\infty\to \Cat^{\perf,\kappa}_\infty[W_{s, \kappa}^{-1}].
\]
Note that $s_{\sC} \times \sD: S_2\sC \times \sD \to \sC \times \sC \times \sD$ is a retract of the map 
$s_{\sC \times \sD}$, so the usual adjointabilty argument shows that $U$ is a finite product-preserving functor. 
Now its $\kappa$-ind completion 
\[
U':\Cat^\perf_\infty\simeq \Ind_{\kappa}(\Cat^{\perf,\kappa}_\infty)\to \Ind_{\kappa}(\Cat^{\perf,\kappa}_\infty[W_{s,\kappa}^{-1}])
\]
is finite product-preserving and it inverts $s_{\sC}$ for all $\sC \in (\Cat^\perf)^\kappa$. 
By Proposition~\ref{prop:additive_loc}, $U'$ is a splitting invariant and any other $\kappa$-finitary splitting invariant 
factors uniquely through $U'$. Thus, $U'$ is a universal splitting invariant valued in an $\infty$-category. 
Now, using \cite[Proposition~5.3.6.2]{HTT}, the composite 
\[
 \Cat^\perf_\infty\xrightarrow{U'} \Ind_{\kappa}(\Cat^{\perf,\kappa}_\infty[W_{s,\kappa}^{-1}]) \to \Ind_{\kappa}(\mathrm{P}_{\coprod}^{\kappa-\mathrm{small}}(\Cat^{\perf,\kappa}_\infty[W_{s,\kappa}^{-1}]))
\]
gives a universal splitting invariant valued in a presentable $\infty$-category, i.e.\ its target is equivalent to $\mathcal{M}_{\splt, \kappa}^\mrm{un}$ (see Example~\ref{exam:universal_additive}). Here, and as in \cite[Proposition~5.3.6.2]{HTT}, $\mrm{P}_{\coprod}^{\kappa-\mathrm{small}}$ denotes the $\infty$-category obtained by freely adding $\kappa$-small colimits while retaining finite coproducts. 

Consequently,  
we obtain that the functor $\Cat^\perf_\infty[W_{s,\kappa}^{-1}] \to \mathcal{M}_{\splt, \kappa}^\mrm{un} \to \mathcal{M}_{\splt, \kappa}$ is 
fully faithful and its image coincides with $\mrm{Rep}\mathcal{M}_{\splt, \kappa}$. 
\end{proof}

\begin{rem}
The immediate analog of our main theorem (Theorem~\ref{thm:main-intro}) is \emph{false} for splitting motives, even the ``connective'' ones. 
This is because one cannot categorify cofibers of splitting motives. 
So the analog that we explore is necessarily different. The point of the theorem above is that while 
$\mathcal U_{\splt,\kappa}$ is not surjective, that is the only obstruction to it being a localization. 
\end{rem}

We may use Corollary~\ref{cor:motadd} together with our localization results of 
Sections~\ref{section:MotLoc}~and~\ref{sec:aleph1case} to reprove the representability of K-theory in $\Motnc$
(see \cite[Theorem~9.8]{blumberg2013universal}) without using the construction of $\Motnc$ as a localization of presheaves.

\begin{cor}\label{cor:Kthcorep}
The unit object $\mathcal{U}_\mrm{loc}(\mrm{Spt}^\omega)$ of $\Motnc$ is compact and we have a natural equivalence
\[
\Maps_{\Motnc}(\mathcal{U}_\mrm{loc}(\mrm{Spt}^\omega), \mathcal{U}_\mrm{loc}(-)) \simeq \K^{\mrm{cn}}(-).
\]
\end{cor}
\begin{proof}
To prove compactness it suffices to show that mapping out of $\mathcal{U}_\mrm{loc}(\mrm{Spt}^\omega)$ commutes with filtered 
colimits. Any filtered diagram lies in $\Motnc^{\kappa}$ for some uncountable regular $\kappa$ 
which by Proposition~\ref{prop:bootstrapdown} is equivalent to $\mathcal{M}_{\kappa}:= \Cat_{\infty}^{\perf,\kappa}[(W_{\mathrm{mot}}^{\kappa})^{-1}]$, so it suffices to show that 
mapping out of $\gamma_{\kappa}(\mrm{Spt}^\omega)$ commutes with $\kappa$-small filtered colimits in 
$\mathcal{M}_\kappa$.

Note that every morphism in $W_{s,\kappa}$ is a motivic equivalence.
Since $\Cat_{\infty}^{\perf,\kappa} \to \mrm{Rep}\mathcal{M}_{\splt,\kappa}$ is a Dwyer--Kan localization at 
$W_{s,\kappa}$ by Corollary~\ref{cor:motadd} and $\Cat_{\infty}^{\perf,\kappa} \to \mathcal{M}_{\kappa}$ is a Dwyer--Kan 
localization at the motivic equivalences, 
it follows that the induced functor
\[
\ell_\kappa : \mrm{Rep}\mathcal{M}_{\splt,\kappa} \to \mathcal{M}_{\kappa} 
\]
is also a Dwyer--Kan localization.
 
Proposition~\ref{prop:additive_morphisms} implies that connective K-theory is corepresented by 
$\mathcal{U}_{\splt,\kappa}(\mrm{Spt}^\omega)$ on $\mrm{Rep}\mathcal{M}_{\splt,\kappa}$.
Since $\K^{\mathrm {cn}} \simeq \tau_{\geq 0}\K$ inverts all motivic equivalences, it factors over the localization $\ell_\kappa$.
In particular, the induced functor on $\Motnc$ is left Kan extended along $\ell_\kappa$, so connective K-theory 
(as a space-valued functor) is corepresented by $\gamma_{\kappa}(\mrm{Spt}^\omega)$ on $\mathcal{M}_{\kappa}$.
 
By \cite[Example~7.9.6 \& Theorem~7.9.8]{cisinski-rl}, every filtered diagram in $\mathcal{M}_{\kappa}$ lifts to a diagram in 
$\Cat_{\infty}^{\perf,\kappa}$, so compactness of $\mathcal{U}_\mrm{loc}(\mrm{Spt}^\omega)$ follows from the fact that 
connective K-theory preserves filtered colimits.
\end{proof}
\begin{rem}
This last corollary, which is already in \cite{blumberg2013universal}, is more subtle than it appears. While most of the 
proof is formal, we crucially use here that the connective cover of non-connective $\K$-theory, defined as the universal 
localizing invariant under the functor corepresented by the unit in $\Cat^\perf_\infty$, coincides with the universal 
splitting invariant (up to idempotent-completion) --- this ingredient is also used in the original proof of this corollary, 
cf.~\cite[Lemma 6.3]{schlichting_negative_2008}.  This is somewhat of a miracle and fails in formally analogous settings, e.g. in the 
setting of $\mathcal{V}$-linear localizing invariants for a general (non-rigid) base $\mathcal{V}$, as in Appendix~\ref{app:linear}. We do not prove this here, but it fails for instance for $\mathcal{V}= \mrm{Spt}^{BG}$, when $G$ is a nontrivial finite group. 

Among other things, that this holds for $\mathcal{V}=\mrm{Spt}$ (or more generally for $\mathcal{V}$ rigid over $\mrm{Spt}$) is crucial in Efimov's approach to the rigidity of motives. 
\end{rem}

\ssec{An explicit class of motivic equivalences}
It follows from Theorem~\ref{thm:Gamma_loop} and \cite{rsw:every-spectrum-is-k-theory} that the objects 
\[
\mathcal{U}_\mrm{loc}(\sC)\;\;\;\; \text{ and }\;\;\;\; \mathcal{U}_\mrm{loc}(\Calk^2\Gamma^2 \sC))
\] are equivalent. 
To finish our preparations we need to construct an actual functor $\eta:\sC \to \Calk^2\Gamma^2\sC$ that realizes such an equivalence.
As in Section~\ref{section:MotLoc}, we will use the abbreviations $\Calk := \Calk(\Spt^\omega)$ and $\Gamma := \Gamma(\Spt^\omega)$.

\begin{lem}\label{lem:eta_spt}
 There exists a functor $\eta:\Spt \to \Calk^{\otimes^2} \otimes \Gamma^{\otimes^2}$ that induces the equivalence 
$\mathcal{U}_\mrm{loc}(\Spt^\omega) \simeq \mathcal{U}_\mrm{loc}(\Calk^{\otimes^2} \otimes \Gamma^{\otimes^2})$. 
\end{lem}
\begin{proof}
By Corollary~\ref{cor:Kthcorep} (see also \cite[Theorem~9.8]{blumberg2013universal}), the map 
\[
\mathcal{U}_\mrm{loc}: 
\Fun_\mrm{ex}(\Spt^\omega, \Calk^{\otimes^2} \otimes \Gamma^{\otimes^2})^\simeq \to \Map_{\Motnc}(\mathcal{U}_\mrm{loc}(\Spt^\omega), \mathcal{U}_\mrm{loc}(\Calk^{\otimes^2} \otimes \Gamma^{\otimes^2})).
\]
induces an equivalence 
\[
\K(\Fun_\mrm{ex}(\Spt^\omega, \Calk^{\otimes^2} \otimes \Gamma^{\otimes^2})) \simeq 
\Map_{\Motnc}(\mathcal{U}_\mrm{loc}(\Spt^\omega), \mathcal{U}_\mrm{loc}(\Calk^{\otimes^2} \otimes \Gamma^{\otimes^2})).
\]
Every connected component of the $\K$-theory space is represented by a functor, so the equivalence 
$\mathcal{U}_\mrm{loc}(\Spt^\omega) \simeq \mathcal{U}_\mrm{loc}(\Calk^{\otimes^2} \otimes \Gamma^{\otimes^2})$ is also represented 
by a functor $\Spt^\omega \to \Calk^{\otimes^2} \otimes \Gamma^{\otimes^2}$.
\end{proof}

\begin{lem}\label{lem:assembly}
There exists a natural map 
\[
 \Calk^{\otimes^2} \otimes \Gamma^{\otimes^2} \otimes \sC \to \Calk^2 (\Gamma^2 \sC)
\] 
that induces an equivalence on any localizing invariant $E$. 
\end{lem}
\begin{proof}
Define $\widetilde\Calk(\sC) = \Calk \otimes \sC$. 
This construction is functorial in $\sC$ and fits into a natural commutative diagram (whose rows are Karoubi sequences)
\[
\begin{tikzcd}
\Spt^\omega \otimes \sC \arrow[r]\arrow[d, "\simeq"] & \Ind(\Spt^\omega)^{\aleph_1} \otimes \sC \arrow[r]
\arrow[d, "\boxtimes"] & \widetilde\Calk(\sC)\arrow[d]\\
\sC \arrow[r] & \Ind(\sC)^{\aleph_1}\arrow[r] & \Calk(\sC).
\end{tikzcd}
\]
In particular, we have a natural transformation of functors $\widetilde\Calk \to \Calk$ and also 
$\widetilde\Calk^2 \to \Calk^2$ which induces an equivalence 
$E(\widetilde\Calk^2(\sC)) \simeq E(\Calk^2(\sC))$ for any localizing invariant $E$. 

By construction of $\Gamma$ (see Section~\ref{ssec:construction_of_gamma}),
we have $\Gamma^2(\sC) = \Gamma^{\otimes^2} \otimes \sC$.
Combining this fact with the construction above, 
we get a natural equivalence
\[
 \Calk^{\otimes^2} \otimes \Gamma^{\otimes^2} \otimes \sC \simeq \widetilde\Calk^2(\Gamma^2(\sC))
 \longrightarrow \Calk^2 (\Gamma^2 \sC)
\] 
as desired.
\end{proof}

Finally, combining Lemma~\ref{lem:eta_spt} and Lemma~\ref{lem:assembly} we obtain the following.

\begin{cor}\label{cor:eta_c}
There are maps 
\[
\eta_\sC:\sC \to \Calk^2(\Gamma^2 \sC)
\]
that become equivalences on $\mathcal{U}_{\mrm{loc}}(-)$.
\end{cor}

\begin{rem}\label{rem:nonfinitary}
Let $\kappa$ be a regular uncountable cardinal. 
All the definitions and constructions in \cite{blumberg2013universal} can be done in the context of $\kappa$-finitary 
localizing invariants. 
The universal $\kappa$-finitary localizing invariant valued in a stable {\it presentable} $\infty$-category 
\[
\mathcal{U}_{\mrm{loc}, \kappa}:\Cat^\perf_\infty \to \mathcal{M}_{\mrm{loc},\kappa}
\]
is a symmetric monoidal functor and the proof of 
\cite[Theorem~9.8]{blumberg2013universal} goes through 
to show that $\mathcal{U}_{\mrm{loc},\kappa}(\Spt^\omega)$ corepresents K-theory (similar to the proof of 
Corollary~\ref{cor:Kthcorep}, this follows from Proposition~\ref{prop:additive_morphisms} using the fact that 
$\mathcal{M}_{\splt, \kappa} \to \mathcal{M}_{\mrm{loc},\kappa}$ is a localization and that connective K-theory inverts 
the corresponding equivalences). 

In particular, the map $\eta:\Spt^\omega \to \Calk^{\otimes^2} \otimes \Gamma^{\otimes^2}$ of Lemma~\ref{lem:eta_spt} 
still induces an equivalence on $\mathcal{U}_{\mrm{loc},\kappa}(-)$. 
Hence it follows that the map $\eta_\sC$ of Corollary~\ref{cor:eta_c} induces an equivalence on any 
$\kappa$-finitary localizing invariant valued in a stable presentable $\infty$-category. 
\end{rem}

\sssec{}
We denote by $W_{\mrm{simple}}$ the smallest class of morphisms in $\Cat^\perf_\infty$ containing 
\begin{itemize}
\item maps of the form $s_{\sC}:S_2\sC \to \sC \times \sC$ for all $\sC \in (\Cat^\perf_\infty)^{\aleph_1}$, 
\item the maps $\eta_\sC:\sC \to \Calk^2\Gamma^2\sC$ constructed in Lemma~\ref{cor:eta_c} 
for all $\sC \in (\Cat^\perf_\infty)^{\aleph_1}$.
\end{itemize}

\begin{lem}\label{lem:aleph1inverts}
Any $\aleph_1$-finitary localizing invariant $E:\Cat^\perf_\infty \to \sX$ valued in a stable $\infty$-category 
inverts all elements of $W_{\mrm{simple}}$.
\end{lem}
\begin{proof}
$E$ inverts the maps $s_{\sC}$ by Proposition~\ref{prop:additive_loc}. 
As observed in Remark~\ref{rem:nonfinitary} any $\aleph_1$-finitary localizing invariant valued in a presentable stable 
$\infty$-category also inverts $\eta_{\sC}$ for all $\sC$, so the result holds if $\sX$ is presentable. 

Now any map $\alpha$ in the class $W_{\mathrm{simple}}$ belongs to $\Cat^{\perf,\kappa}_\infty$ for some regular cardinal 
$\kappa$. 
Set $\sX'$ to be the smallest full subcategory of $\sX$ containing the image of $E|_{\Cat^{\perf,\kappa}_\infty}$ and closed 
under finite limits and finite colimits. Now using \cite[Proposition~5.3.6.2]{HTT} define $\bar{\sX}$ to be the universal 
closure of $\sX'$ under $\kappa$-small colimits, i.e. 
\[
\bar{\sX} = \mathrm{P}_{\ex}^{\kappa-\mathrm{small}}(\sX').
\]
Now the Kan extension
\[
\bar{E}:\Cat^\perf_\infty \simeq \Ind_{\kappa}(\Cat^{\perf,\kappa}_\infty) \to \Ind_{\kappa}(\bar{\sX})
\]
of the restriction $E|_{\Cat^{\perf,\kappa}_\infty}: \Cat^{\perf,\kappa}_\infty \to \sX' \to \bar{\sX}$
is a $\kappa$-finitary localizing invariant valued in a presentable $\infty$-category and so it inverts $\alpha$. Since 
$\sX' \to \bar{\sX}$ is conservative, we obtain that $E$ also inverts $\alpha$.
\end{proof}

Being a finitary localizing invariant, the functor $\mathcal{U}_\mrm{loc}$ also inverts all elements of $W_{\mrm{simple}}$, 
and for any $\infty$-category $\sX$ we have an embedding 
\[
\Theta:\Fun^{\aleph_1-\mathrm{fin}}(\Motnc,\sX) \to \Fun^{\aleph_1-\mathrm{fin}}_{W_{\mrm{simple}}}(\Cat^\perf_\infty,\sX)
\]
where the subscript $W_{\mrm{simple}}$ denotes the subcategory of functors inverting all elements of $W_{\mrm{simple}}$.

\begin{thm}\label{thm:localization_simple}
For any $\infty$-category $\sX$, the functor $\Theta$ is an equivalence of $\infty$-categories.
\end{thm}

We note that Theorem~\ref{thm:universal_omega_1_finitary} follows immediately from Theorem~\ref{thm:localization_simple}. 
Indeed, since any fiber sequence in $\Motnc$ can be modelled by a Karoubi sequence in $\Cat^\perf_\infty$ we also have 
that the embedding 
\[
\Fun^{\aleph_1-\mathrm{fin}}_{\mathrm{ex}}(\Motnc,\sX) \to \Fun^{\mathrm{loc},\aleph_1-\mathrm{fin}}_{W_{\mrm{simple}}}(\Cat^\perf_\infty,\sX)
\]
is an equivalence. 
By Corollary~\ref{cor:eta_c} (see also Remark~\ref{rem:nonfinitary}) and definition of $W_{\mrm{simple}}$, 
its elements are inverted by any $\aleph_1$-finitary localizing invariant $E$ valued in a stable  
$\infty$-category, so we have a chain of fully faithful functors 
\[
\Fun^{\aleph_1-\mathrm{fin}}_{\mathrm{ex}}(\Motnc,\sX)\to \Fun^{\mathrm{loc}, \aleph_1-\mathrm{fin}}(\Cat^\perf_\infty,\sX) \stackrel{=}\to \Fun^{\mathrm{loc}, \aleph_1-\mathrm{fin}}_{W_{\mrm{simple}}}(\Cat^\perf_\infty,\sX) 
\]
each of them is now forced to be an equivalence.

The main ingredient in the proof is the following {\it currently unpublished} result of Efimov.

\begin{thm}\label{thm:efimov_morphisms}
Let $\sC$ be an $\aleph_1$-compact stable $\infty$-category. 
Then the induced map
\[
\Map_{\mathcal{M}_{\splt, \aleph_1}}(\mathcal{U}_{\splt,\aleph_1}(\sC), \mathcal{U}_{\splt,\aleph_1}(\mrm{Calk}^2(\sD))) \to \Map_{\Motnc}(\mathcal{U}_{\mrm{loc}}(\sC), \mathcal{U}_{\mrm{loc}}(\mrm{Calk}^2(\sD)))
\]
is an equivalence for any $\sD \in \Cat^\perf_\infty$. 
\end{thm}

\begin{cor}\label{cor:inverting_unit}
Let $\kappa$ be a regular cardinal. 
Let $\sC, \sD$ be small idempotent complete stable $\infty$-categories such that $\sC$ is also $\aleph_1$-compact. 
Given exact functors $F,G:\sC \to \Calk^2\sD$ such that there is an equivalence 
$\mathcal{U}_\mrm{loc}(F) \simeq \mathcal{U}_\mrm{loc}(G)$, there is also an equivalence $\mathcal{U}_{\splt,\kappa}(F) \simeq \mathcal{U}_{\splt,\kappa}(G)$. 
\end{cor}
\begin{proof}
Since the map $\mathcal{M}_{\splt, \kappa} \to \Motnc$ factors through $\mathcal{M}_{\splt, \omega}$, it suffices to 
prove the case $\kappa>\omega$.
For any $\kappa \ge \aleph_1$ we have a commutative diagram 
\[
\begin{tikzcd}
\Maps_{\mathcal{M}_{\splt, \kappa}}(\mathcal{U}_{\splt,\kappa}(\sC),\mathcal{U}_{\splt,\kappa}(\Calk^2\sD))\arrow[d]\arrow[dr]\\
\Maps_{\mathcal{M}_{\splt, {\aleph_1}}}(\mathcal{U}_{\splt,{\aleph_1}}(\sC),\mathcal{U}_{\splt,{\aleph_1}}(\Calk^2\sD))\arrow[r] &
\Maps_{\Motnc}(\mathcal{U}_{\mrm{loc}}(\sC),\mathcal{U}_{\mrm{loc}}(\Calk^2\sD)).
\end{tikzcd}
\]
The horizontal map is an equivalence by Theorem~\ref{thm:efimov_morphisms} and the vertical map is an equivalence 
by Proposition~\ref{prop:additive_morphisms}. Hence, the diagonal map is also an equivalence and the result follows from 
of its $\pi_0$-injectivity.
\end{proof}

\begin{cor}\label{cor:additiveinvCalk2}
Let $E:\Cat^\perf_\infty \to \sX$ be a functor that 
inverts maps $S_2\sE \to \sE \times \sE$ for all $\sE \in \Cat^\perf_\infty$. 
Let $\sC, \sD$ be small idempotent complete stable 
$\infty$-categories such that $\sC$ is $\aleph_1$-compact. 
Then given exact functors $F,G:\sC \to \Calk^2\sD$ such that there is an equivalence 
$\mathcal{U}_\mrm{loc}(F) \simeq \mathcal{U}_\mrm{loc}(G)$, there is also an equivalence 
$E(F) \simeq E(G)$. 
\end{cor}
\begin{proof}
Since being $\kappa$-finitary implies being $\lambda$-finitary for $\lambda \geq \kappa$, we may assume $\kappa$ is large 
enough, so that $\sC,\sD,\Calk^2(\sC), \Calk^2(\sD)$ are all $\kappa$-compact. 
By Corollary~\ref{cor:motadd}, the restriction of $E$ to $\kappa$-compact objects factors through 
$\mathrm{Rep}\mathcal{M}_{\splt, \kappa}$ and so the result follows from 
Corollary~\ref{cor:inverting_unit}. 
\end{proof}

\begin{proof}[Proof of Theorem~\ref{thm:localization_simple}]
By Theorem~\ref{thm:main-intro}, it suffices to show that any $\aleph_1$-filtered colimit-preserving functor 
$E:\Cat^\perf_\infty \to \sX$ inverts motivic equivalences as long as it inverts all elements of $W_{\mrm{simple}}$. 
By Proposition~\ref{prop:bootstrapdown} it suffices to show that $E$ inverts 
any motivic equivalence $f:\sC \to \sD$ where we assume that $\sC$ and $\sD$ are $\aleph_1$-compact. 

Consider the solid part of the commutative diagram
\[
\begin{tikzcd}[column sep=3.5em]
\sC\arrow[r,"f"]\arrow[d,"\eta_\sC"'] & \sD\arrow[d, "\eta_\sD"]\arrow[dl, dotted]\\
\Calk^2\Gamma^2\sC\arrow[r,"\Calk^2\Gamma^2 f"] & \Calk^2\Gamma^2\sD.
\end{tikzcd}
\]
Since the maps $\eta$ are invertible in $\Motnc$, there is a dotted map 
\[
\begin{tikzcd}[column sep=3.5em]
\mathcal{U}_{\mrm{loc}}(\sC)\arrow[r,"f"]\arrow[d,"\eta_\sC"'] & \mathcal{U}_{\mrm{loc}}(\sD)\arrow[d, "\eta_\sD"]\arrow[dl, dotted]\\
\mathcal{U}_{\mrm{loc}}(\Calk^2\Gamma^2\sC)\arrow[r,"\Calk^2\Gamma^2 f"] & \mathcal{U}_{\mrm{loc}}(\Calk^2\Gamma^2\sD)
\end{tikzcd}
\]
making the diagram in $\Motnc$ commutative. By Theorem~\ref{thm:efimov_morphisms} and by the virtue of the fact that any 
class in $\K_0$ of a stable $\infty$-category is represented by an object, there exists a functor 
$\mrm{ef}:\sD \to \mrm{Calk}^2\Gamma^2\sC$ that induces the dotted arrow. 
We note that $E$ at least inverts splitting equivalences, so we may apply Corollary~\ref{cor:additiveinvCalk2}: 
$E(\mrm{ef}) \circ E(f)$ and 
$E(\Calk^2\Gamma^2 f) \circ E(\mrm{ef})$ are homotopic to equivalences $E(\eta_\sC)$ and $E(\eta_\sD)$, and so $E(f)$ is also 
an equivalence. 
\end{proof}

As an application of this result, we obtain the following.

\begin{thm}\label{thm:countable_products}
The universal localizing invariant 
\[
\mathcal{U}_\mrm{loc}:\Cat^\perf_\infty \to \Motnc
\]
preserves countable products. 
\end{thm}
\begin{proof}
It suffices to check that the adjunction 
\[
\Delta : \Cat^\perf_\infty  \rightleftarrows  (\prod_J \Cat^\perf_\infty) : \Pi
\]
factors through an adjunction on motives.
By Theorem~\ref{thm:main-intro}, we may identify $\Motnc$ with the localization 
$\Cat^\perf_\infty[W^{-1}_{\mrm{mot}}]$ where $W_{\mrm{mot}}$ denotes the class of motivic equivalences. Since $\Cat^\perf_\infty$ is a category of cofibrant objects, the functor
$\prod_{\mbf{N}} \Cat^\perf_\infty \to \prod_{\mbf{N}} \Motnc$ is also a localization at the class of levelwise 
motivic equivalences (one can see this directly or using \cite[Theorem~7.9.8]{cisinski-rl}). 
Now it suffices to show that for any finitary localizing invariant $E$ and a countable family of motivic equivalences 
$f_i: \sC_i \to \sD_i \in \Cat^{\perf}_\infty$ the map
\[
E(\prod\limits_{\mbf{N}} \sC_i) \to E(\prod\limits_{\mbf{N}} \sD_i)
\]
is an equivalence. 
Consider the functor $E^{\Pi}(\sC) := E(\prod_{\mbf{N}} \sC)$. 
The map above is a retract of $E^\Pi(-)$ applied to the motivic equivalence 
$\bigoplus_\mbf{N} \sC_i \to \bigoplus_\mbf{N} \sD_i$, so it suffices to show that $E^{\Pi}(-)$ sends motivic 
equivalences to equivalences. This follows from Theorem~\ref{thm:localization_simple} since $E^{\Pi}$ is an 
$\aleph_1$-finitary localizing invariant.
\end{proof}

\appendix
\section{Reminder on the Grayson construction}\label{sec:reminder_on_the_grayson_construction}

Kasprowski and the third author \cite{KasprowskiWinges} defined for a small stable $\infty$-category $\sC$ 
a new stable $\infty$-category $\Gamma \sC$ which serves as a categorical delooping of $\sC$. 
More precisely, it has the property that $\mathcal{U}_\mathrm{loc}(\Gamma \sC)$ is equivalent to the 
$\Omega \mathcal{U}_\mathrm{loc}(\sC)$. 
In this section we give a slightly different presentation of the construction which we learnt from Sasha Efimov. 
Our proof of this presentation is based on the weighted $\mathbb{A}^1$-invariance of Annala--Hoyois--Iwasa 
\cite{annala2023algebraic}. 
Using this interpretation, one can give a new proof that  $\Gamma \sC$, somewhat surprisingly, also deloops arbitrary 
localizing invariants which are not assumed to be finitary or to factor through 
$\Motnc$. 

\sssec{}
First, we need to introduce several notions from spectral algebraic geometry. 
We denote by $\mathrm{CAlg}^{\mrm{cn}}$ the $\infty$-category of connective commutative ring spectra. 

Given an accessible functor $F:\mathrm{CAlg}^{\mrm{cn}} \to \Spc$ one defines the $\infty$-category $\mbf{Perf}_F$ as the 
value of the right Kan extension of 
\[
\mbf{Perf}_{(-)}:(\mathrm{CAlg}^{\mrm{cn}})^{\mrm{op}} \to \Cat^\perf
\]
along the Yoneda embedding $\mrm{Spec}:(\mathrm{CAlg}^{\mrm{cn}})^{\mrm{op}} \to \Fun(\mathrm{CAlg}^{\mrm{cn}}, \Spc)$. 

\begin{exam}\label{exam:perf_scheme}
For a spectral scheme $X$, the $\infty$-category $\mbf{Perf}_X$ coincides with the usual $\infty$-category of 
perfect complexes defined in terms of sheaves of $\mathcal{O}_X$-modules (see \cite[Propositions~6.2.4.1~and~6.2.5.10]{SAG}). 
\end{exam}

\begin{exam}\label{exam:quot_stack}
Consider the functors $\mathrm{CAlg}^{\mrm{cn}} \to \Spc$ corepresented by the free monoid algebras of $\mbf{N}$ and 
$\mbf{Z}$:
\[
\mbf{A}^1 = \mrm{Spec}(\mbf{S}[\mbf{N}]), \; \mbf{G}_{m} = \mrm{Spec}(\mbf{S}[\mbf{Z}]).
\]
The comultiplication map $\mbf{Z} \to \mbf{Z} \times \mbf{Z}$ and the coaction map $\mbf{N} \to \mbf{N} \times \mbf{Z}$ given by the same formula
\[
n \mapsto (n,n)
\] 
yield simplicial diagrams in $\Fun(\mathrm{CAlg}^{\mrm{cn}}, \Spc)$:
\[
\begin{tikzcd}[row sep = 1]
\mbf{A}^1 \arrow[r, shift right] & \mbf{A}^1 \times \mbf{G}_{m} 
\arrow[l] \arrow[l, shift right] \arrow[r, shift right] \arrow[r, shift right=2] 
&\mbf{A}^1 \times \mbf{G}_{m} \times \mbf{G}_{m}
\arrow[l]\arrow[l, shift right=1]\arrow[l, shift right=2] \arrow[r, shift right]\arrow[r, shift right=2]\arrow[r, shift right=3] 
&\cdots
\arrow[l]\arrow[l, shift right=1]\arrow[l, shift right=2]\arrow[l, shift right=3] \\
\Spec(k) \arrow[r, shift right] & \mbf{G}_{m} 
\arrow[l] \arrow[l, shift right] \arrow[r, shift right] \arrow[r, shift right=2] 
&\mbf{G}_{m} \times \mbf{G}_{m}
\arrow[l]\arrow[l, shift right=1]\arrow[l, shift right=2] \arrow[r, shift right]\arrow[r, shift right=2]\arrow[r, shift right=3] 
&\cdots
\arrow[l]\arrow[l, shift right=1]\arrow[l, shift right=2]\arrow[l, shift right=3] 
\end{tikzcd}
\]
We denote their respective colimits by $[\mbf{A}^1/\mbf{G}_{m}]$ and $\mrm{B}\mbf{G}_{m}$ 
(see also \cite[Sections~2,4]{Moulinos_2021}). 
The main result in \cite{Moulinos_2021} (combined with \cite[Proposition~9.4.4.5]{SAG}) shows that we have 
\[
\mbf{Perf}_{[\mbf{A}^1/\mbf{G}_{m}]} \simeq \Fun((\mbf{Z},\leq), \Spt)^\omega,
\; \; \; \; \text{ and } \; \; \; \; 
\mbf{Perf}_{B\mbf{G}_{m}} \simeq \Fun(\mbf{Z}^{\mrm{disc}}, \Spt)^\omega.
\]
\end{exam}

\begin{warning}
Usually when employing the functor-of-points approach as above, one limits oneself to {\it fpqc sheaves} among
arbitrary functors. To make the exposition shorter and more elementary we avoid using the topology and 
perform all the constructions directly with functors. 
In particular, $[\mbf{A}^1/\mbf{G}_{m}]$ and $\mrm{B}\mbf{G}_{m}$ in Example~\ref{exam:quot_stack} 
are defined as colimits of functors and not of fpqc sheaves. These two different definitions, however, do not change the 
values of $\mathrm{Perf}_{(-)}$ \cite[Propositions~6.2.3.1~and~6.2.6.2]{SAG}. 
\end{warning}

We can make the formula given in Example~\ref{exam:quot_stack} more explicit.

\begin{prop}\label{prop:Z-complexes}
Let $\sC$ be a small idempotent complete stable $\infty$-category. The following are equivalent for an object 
$X$ of $\Fun((\mbf{Z},\le), \Ind(\sC))$:
\begin{enumerate}
\item $X$ is compact;
\item $X(i)$ are in $\sC$, $X(i) = 0$ for small enough $i$, and the sequence stabilizes for big enough $i$. 
\end{enumerate}
\end{prop}
\begin{proof}
Given $M \in \Ind(\sC)$ let $X_{M,k}$ be the object of $\Fun((\mbf{Z},\le), \Ind(\sC))$ such that $X_{M,k}(i)$ is non-zero in one degree $k$ only and $X_{M,k}(k)=M$. This 
corepresents the functor 
\[
F \in \Fun((\mbf{Z},\le), \Ind(\sC)) \longmapsto \Fib(\Map(M, F(k)) \to \Map(M, F(k+1)))
\]
which commutes with filtered colimits whenever $M$ belongs to $\sC$. 
In the same context we define $Y_M$ to be the object of $\Fun((\mbf{Z},\le), \Ind(\sC))$ such that $Y_M(i) = 0$ for $i<0$, $Y_M(0) = M$ and 
$Y_M(i) \to Y_M(i+1)$ is an equivalence for all $i\ge 0$. This corepresents the functor
\[
F \in \Fun((\mbf{Z},\le), \Ind(\sC)) \longmapsto \Map(M, F(0))
\]
which commutes with filtered colimits whenever $M$ is compact in $\Ind(\sC)$. 
Any object $X$ satisfying the conditions in (2) is a finite extension of objects of the form $Y_M$ and $X_{M,k}$ for 
compact $M$, hence it is compact by the computations above. 

On the other hand, for a compact object $X$ of $\Fun((\mbf{Z},\le), \Ind(\sC))$, consider 
\[
X_n(i) = 
\begin{cases}
0,  &\text{ if } i<-n,\\
X(i), & \text{ if } -n \ge i\le n\\
X(n), & \text{ if } i>n
\end{cases}
\]
We have a direct system of maps  
\[
X_0 \to \cdots \to X_n \to X_{n+1} \to \cdots 
\]
whose colimit is $X$. By compactness, the identity map on $X$ factors through some $X_n$, which satisfies condition (2) 
along with any of its retracts.
\end{proof}

\sssec{}
We will now explore Example~\ref{exam:quot_stack} in more detail. 
Note first that we have a projection map and an immersion map 
\[
\pi:[\mbf{A}^1/\mbf{G}_{m}] \to \mrm{B}\mbf{G}_{m} \; \; \;  \text{ and } \; \; \; \iota:\mrm{B}\mbf{G}_{m} \to [\mbf{A}^1/\mbf{G}_{m}]
\]
with $\pi \circ \iota \simeq \id_{\mrm{B}\mbf{G}_{m}}$.
Using the comultiplication map $\mbf{N}\to \mbf{N} \times \mbf{N}$ and the fact that sifted colimits commute with 
products, we may construct a multiplication map
\[
m:[\mbf{A}^1/\mbf{G}_{m}] \times [\mbf{A}^1/\mbf{G}_{m}] \to [\mbf{A}^1/\mbf{G}_{m}] \times \mrm{B}\mbf{G}_{m} \to [\mbf{A}^1/\mbf{G}_{m}]
\]
such that the two composites
\[
\begin{tikzcd}
{[\mbf{A}^1/\mbf{G}_{m}]} \arrow[r,shift left, "0"]\arrow[r,shift right, swap, "1"] 
&{[\mbf{A}^1/\mbf{G}_{m}] \times [\mbf{A}^1/\mbf{G}_{m}]} \arrow[r, "m"] & {[\mbf{A}^1/\mbf{G}_{m}]}
\end{tikzcd}
\]
are equivalent to $\id_{[\mbf{A}^1/\mbf{G}_{m}]}$ and $\iota \circ \pi$.
We also have a commutative diagram 
\begin{equation}\label{eq:P1}
\begin{tikzcd}
{\mbf{G}_{m}} \arrow[r]\arrow[d] & {\mbf{A}^1}\arrow[d,"1"]\\
{\mbf{A}^1} \arrow[r, "\id"] & {[\mbf{A}^1/\mbf{G}_{m}]}
\end{tikzcd}
\end{equation}
which yields a map 
\[
\phi:\mbf{P}^1 \to [\mbf{A}^1/\mbf{G}_{m}],
\]
where $\mbf{P}^1$ denotes the pushout\footnote{taken in $\Fun(\CAlg^\mrm{cn}, \Spc)$} of the diagram (\ref{eq:P1}).

Additionally, we define a stack $\mathcal{B}$ as the pushout  
\[
\begin{tikzcd}
{\mrm{B}\mbf{G}_{m}} \arrow[r,"\iota"]\arrow[d,"\iota"] & {[\mbf{A}^1/\mbf{G}_{m}]}\arrow[d] \\
{[\mbf{A}^1/\mbf{G}_{m}]}\arrow[r] & {\mathcal{B}}.
\end{tikzcd}
\]
Again we have a projection map and an immersion map
\[
\pi_{\mathcal{B}}:\mathcal{B} \to \mrm{B}\mbf{G}_{m} \; \; \;  \text{ and } \; \; \; \iota_{\mathcal{B}}:\mrm{B}\mbf{G}_{m} \to \mathcal{B}
\]
with $\pi_{\mathcal{B}} \circ \iota_{\mathcal{B}} \simeq \id_{\mrm{B}\mbf{G}_{m}}$. Similarly, we have an action map 
\[
a:[\mbf{A}^1/\mbf{G}_{m}] \times \mathcal{B} \to  \mathcal{B}
\]
such that the two composites
\[
\begin{tikzcd}
{\mathcal{B}} \arrow[r,shift left, "0"]\arrow[r,shift right, swap, "1"] 
&{[\mbf{A}^1/\mbf{G}_{m}]} \times \mathcal{B} \arrow[r, "a"] & {\mathcal{B}}
\end{tikzcd}
\]
are equivalent to $\id_{\mathcal{B}}$ and $\iota \circ \pi$.
There is also a map $h:\sB \to [\mbf{A}^1/\mbf{G}_{m}]$ induced by the commutative diagram 
\[
\begin{tikzcd}
{\mrm{B}\mbf{G}_{m}} \arrow[r,"\iota"]\arrow[d,"\iota"] & {[\mbf{A}^1/\mbf{G}_{m}]}\arrow[d,"\id"] \\
{[\mbf{A}^1/\mbf{G}_{m}]}\arrow[r,"\id"] & {[\mbf{A}^1/\mbf{G}_{m}]}.
\end{tikzcd}
\]

Unraveling the definitions of all the maps we obtain the following lemma.

\begin{lem}\label{lem:P-homotopy-diagram}
There are commutative diagrams of the following form:
\begin{enumerate}
\item
$
\begin{tikzcd}
{[\mbf{A}^1/\mbf{G}_{m}]}\arrow[rr, "\id"] \arrow[d, shift left=2, "0"]\arrow[d,shift right=2, swap, "1"]
&&{[\mbf{A}^1/\mbf{G}_{m}]} \arrow[d, shift left = 2, "\pi \circ \iota"]\arrow[d, shift right = 2, swap, "\id"]\\ 
\mbf{P}^1 \times {[\mbf{A}^1/\mbf{G}_{m}]}\arrow[r, "\phi \times \id"]& 
{[\mbf{A}^1/\mbf{G}_{m}] \times [\mbf{A}^1/\mbf{G}_{m}]} \arrow[r,"m"] &{[\mbf{A}^1/\mbf{G}_{m}]}
\end{tikzcd}
$
\item
$
\begin{tikzcd}
\mathcal{B} \arrow[rr, "\id"] \arrow[d, shift left=2, "0"]\arrow[d,shift right=2, swap, "1"]
&&\mathcal{B} \arrow[d, shift left = 2, "\pi \circ \iota"]\arrow[d, shift right = 2, swap, "\id"]\\ 
\mbf{P}^1 \times \mathcal{B} \arrow[r, "\phi \times \id"]& 
{[\mbf{A}^1/\mbf{G}_{m}] \times \mathcal{B}} \arrow[r,"a"] &\mathcal{B}
\end{tikzcd}
$
\end{enumerate}
\end{lem}

\ssec{Construction of \texorpdfstring{$\Gamma$}{Gamma}}\label{ssec:construction_of_gamma}

By \cite[Proposition~6.3]{Moulinos_2021}, there is a unique map 
\[
1:\mrm{Spec}(\mbf{S}) \to [\mbf{A}^1/\mbf{G}_{m}] 
\]
which induces the colimit functor
\[
\Fun((\mbf{Z},\le),\Spt)^\omega \simeq \mbf{Perf}_{[\mbf{A}^1/\mbf{G}_{m}]} \to \Spt^\omega.
\]
Thanks to Proposition~\ref{prop:Z-complexes}, the functor identifies the codomain with the localization of 
$\mbf{Perf}_{[\mbf{A}^1/\mbf{G}_{m}]}$ 
at the subcategory of those $F \in \Fun((\mbf{Z},\le), \Spt)^\omega$ for which $F(i)$ is nonzero for only 
finitely many $i$. We denote this subcategory by $\mbf{Perf}_{0, [\mbf{A}^1/\mbf{G}_{m}]}$. 
By Proposition~\ref{prop:Z-complexes} and by construction of $\mbf{Perf}_{(-)}$, we may identify  $\mbf{Perf}_{\sB}$ with the 
full subcategory of 
\[
(F,G)\in \mrm{Fun}((\mbf{Z}, \leq), \Spt) \times_{\mrm{Fun}(\mbf{Z}^\mrm{disc}, \Spt)} \mrm{Fun}((\mbf{Z}, \leq), \Spt)
\]
such that $F(i)$ and $G(i)$ belong to $\Spt^\omega$ for all $i$, $F(i)$ and $G(i)$ are 0 for small enough $i$, and the 
sequences given by $F(i)$ and $G(i)$ stabilize for big enough $i$. 
Similarly to the above, the map 
\[
\mrm{Spec}(\mbf{S})\coprod \mrm{Spec}(\mbf{S}) \to [\mbf{A}^1/\mbf{G}_{m}] \coprod [\mbf{A}^1/\mbf{G}_{m}] \to \sB
\] 
also induces a localization functor
\[
\mbf{Perf}_{\sB} \to \Spt^\omega \times \Spt^\omega
\]
whose kernel is denoted by $\mbf{Perf}_{0, \mathcal{B}}$. 
The commutative diagram 
\[
\begin{tikzcd}
\mrm{Spec}(\mbf{S})\coprod \mrm{Spec}(\mbf{S}) \arrow[r]\arrow[d] & \sB\arrow[d, "h"]\\
\mrm{Spec}(\mbf{S}) \arrow[r] & {[\mbf{A}^1/\mbf{G}_{m}]}
\end{tikzcd}
\]
induces a map 
\[
\Delta:\mbf{Perf}_{0, [\mbf{A}^1/\mbf{G}_{m}]} \to \mbf{Perf}_{0, \mathcal{B}}.
\]
Now we may construct a fully faithful functor 
\[
\Delta':\mbf{Perf}_{0, [\mbf{A}^1/\mbf{G}_{m}]}
\to
\mbf{Perf}_{0, \mathcal{B}} \vec{\times}_{\Ind(\mbf{Perf}_{0, \mathcal{\sB}})^{\aleph_1}} \Ind(\mbf{Perf}_{0, [\mbf{A}^1/\mbf{G}_{m}]})^{\aleph_1}
\]
Given a stable $\infty$-category $\sC$ we denote by $\Gamma \sC$ the Karoubi quotient of 
\[
\Delta' \otimes \sC: \mbf{Perf}_{0, [\mbf{A}^1/\mbf{G}_{m}]} \otimes \sC
\to
(\mbf{Perf}_{0, \mathcal{B}} \vec{\times}_{\Ind(\mbf{Perf}_{0, \mathcal{\sB}})^{\aleph_1}} \Ind(\mbf{Perf}_{0, [\mbf{A}^1/\mbf{G}_{m}]})^{\aleph_1}) \otimes \sC.
\]

\begin{rem}\label{rem:support}
The notation for the subcategories comes from the theory of quasi-geometric algebraic stacks. 
More generally, there is a notion of an open immersion of functors $\mathrm{CAlg}^{\mrm{cn}} \to \Spc$ and for an open immersion 
$U\to X$ of those functors that are {\it quasi-geometric} (see \cite[Definition~9.1.0.1]{SAG}) 
and {\it perfect} (see \cite[Definition~9.4.4.1]{SAG}) one always has that 
$\mbf{Perf}_X \to \mbf{Perf}_U$ is a homological epimorphism and the kernel consists of perfect complexes 
supported on the complement to $U$. 
\end{rem}

\begin{rem}\label{rem:binacyc}
Example~\ref{exam:quot_stack} together with Proposition~\ref{prop:Z-complexes} shows that 
$\mbf{Perf}_{[\mbf{A}^1/\mbf{G}_{m}]}$ is equivalent to the $\infty$-category of $\mathbb{Z}$-complexes 
$\mathcal{F}(\Spt^\omega)$ of 
\cite{KasprowskiWinges}. Similarly, one can show that $\mbf{Perf}_{\mathcal{B}}$ is equivalent to 
$\mathcal{B}(\Spt^\omega)$, the $\infty$-category of binary complexes over $\Spt^\omega$. 

Futhermore, $\mbf{Perf}_{0, [\mbf{A}^1/\mbf{G}_{m}]}$ and $\mbf{Perf}_{0, \mathcal{B}}$ correspond to 
the acyclic versions $\mathcal{F}^q(\Spt^\omega)$ and $\mathcal{B}^q(\Spt^\omega)$.
\end{rem}

The following is essentially \cite[Corollary~4.8]{annala2023algebraic}, but done in the context of 
spectral stacks (rather than derived schemes). 

\begin{prop}\label{prop:localizing-isomorphism}
For any localizing invariant $E$ the maps 
\[
\iota^*:E(\mbf{Perf}_{[\mbf{A}^1/\mbf{G}_{m}]}) \to E(\mbf{Perf}_{\mrm{B}\mbf{G}_{m}})
\]
\[
\iota_{\mathcal{B}}^*:E(\mbf{Perf}_{\mathcal{B}}) \to E(\mbf{Perf}_{\mrm{B}\mbf{G}_{m}})
\]
are equivalences. 
\end{prop}
\begin{proof}
Recall from \cite[Theorem~7.2.2.1]{SAG} that for any localizing invariant $F$
\[
F(\Spt^\omega) \oplus F(\Spt^\omega) \stackrel{p^* + i_*}\to F(\mbf{Perf}_{\mbf{P}^1}) 
\]
where $p:\mbf{P}^1 \to \Spec (\mbf{S})$ is the projection map and $i:\Spec(\mbf{S}) \to \mbf{P}^1$ is the embedding of 
the point at infinity. Hence, in particular, the maps induced by the embeddings of 0 and 1 in $\mbf{P}^1$ 
\[
0^*,1^*:F(\mbf{Perf}_{\mbf{P}^1}) \to F(\Spt^\omega)
\]
are homotopic. 
Taking $F(\sC) := E(\sC \otimes \mbf{Perf}_{[\mbf{A}^1/\mbf{G}_{m}]})$ or 
$F(\sC) := E(\sC \otimes \mbf{Perf}_{\mathcal{B}})$, respectively, and applying 
Lemma~\ref{lem:P-homotopy-diagram}, we obtain that the maps 
\[
\pi^*\circ \iota^* \simeq (\iota \circ \pi)^*:E(\mbf{Perf}_{[\mbf{A}^1/\mbf{G}_{m}]}) \to E(\mbf{Perf}_{[\mbf{A}^1/\mbf{G}_{m}]})
\]
\[
\pi_{\mathcal{B}}^*\circ \iota_{\mathcal{B}}^* \simeq (\iota_{\mathcal{B}} \circ \pi_{\mathcal{B}})^*:E(\mbf{Perf}_{\mathcal{B}}) \to E(\mbf{Perf}_{\mathcal{B}})
\]
are homotopic to the identity maps. On the other hand, $\iota^* \circ \pi^*$ and 
$\iota_{\mathcal{B}}^* \circ \pi^*_{\mathcal{B}}$ are homotopic to the identity by construction. Hence $\iota^*$ and 
$\pi^*$ (resp., $\iota_{\mathcal{B}}^*$ and $\pi_{\mathcal{B}}^*$) are mutually inverse equivalences.
\end{proof}

\begin{thm}\label{thm:Gamma_loop}
For any small idempotent complete stable $\infty$-category $\sC$ 
there is a sequence in $\Cat^{\perf}_\infty$
\[
\Gamma \sC \stackrel{\alpha}\to A\sC \stackrel{\beta}\to I\sC
\]
functorial in $\sC$ equipped with a natural transformation $i:\sC \to I\sC$ such that for any localizing invariant $E$ valued 
in a stable $\infty$-category the following conditions hold: 
\begin{itemize}
\item the composite $\beta\circ \alpha$ is zero and $E$ applied to the sequence is a cofiber sequence;
\item $E(A\sC) \simeq 0$;
\item $E(i):E(\sC) \to E(I\sC)$ is an equivalence.
\end{itemize}
In particular, there is a functorial identification $E(\Gamma\sC) \simeq \Omega E(\sC)$.
\end{thm}
\begin{proof}
By construction, $E(\Gamma\sC)$ is canonically equivalent to the cofiber 
\[
\Cofib(\Delta:E(\mbf{Perf}_{0, [\mbf{A}^1/\mbf{G}_{m}]} \otimes \sC) \to E(\mbf{Perf}_{0, \mathcal{B}} \otimes \sC).
\]
By construction of $\Gamma$, we have a functorial commutative diagram 
\[
\begin{tikzcd}{}
\mbf{Perf}_{0, [\mbf{A}^1/\mbf{G}_{m}]} \otimes \sC) \arrow[r]\arrow[d, "\Delta"]
& \mbf{Perf}_{[\mbf{A}^1/\mbf{G}_{m}]} \otimes \sC) \arrow[r]\arrow[d, "h^*"] 
& \Spt^\omega \otimes \sC \arrow[d, "\mathrm{diag}"] \\
\mbf{Perf}_{0, \mathcal{B}} \otimes \sC \arrow[r]\arrow[d]
& \mbf{Perf}_{\mathcal{B}} \otimes \sC \arrow[r]\arrow[d]
& (\Spt^\omega \otimes \sC) \times (\Spt^\omega \otimes_k \sC)\arrow[d]\\
\Cone(\Delta)\otimes \sC \arrow[r]
& \Cone(h^*) \otimes \sC \arrow[r] 
& \Cone(\mathrm{diag})\otimes \sC
\end{tikzcd}
\]
where $\Cone(f:\sA \to \sB)$ denotes the Karoubi quotient $\sA \times^\to_{\Ind(\sB)^{\aleph_1}} \Ind(\sB)^{\aleph_1}$. 
The bottom sequence composes to 0 by construction. 
The top two horizontal sequences are localization sequences, so they induce cofiber sequences on $E$. On the other 
hand, all vertical sequences also induce cofiber sequences on $E$ by Proposition~\ref{prop:cat-perf-cofibration-cat}. 
Hence, the bottom sequence also induces a cofiber sequence on $E$. 

We set $A\sC = \Cone(h^*) \otimes \sC$ and $I\sC = \Cone(\mathrm{diag})\otimes \sC$. Define $i$ to be the composite 
\[
\sC \stackrel{(0, \id)}\to
(\Spt^\omega \otimes \sC) \times (\Spt^\omega \otimes \sC) \to \Cone(\mathrm{diag}) \otimes \sC.
\]
The map $E(h^*)$ is an equivalence by Proposition~\ref{prop:localizing-isomorphism}, so $E(A\sC) \simeq 0$. 
Being a cofiber of the diagonal map into the direct sum, the object $E(\Cone(\mathrm{diag})\otimes \sC)$ may be identified 
with $E(\sC)$ with the projection map being  
\[
E(\sC \times \sC) \stackrel{\begin{pmatrix} -1 & 1\end{pmatrix}}\to E(\sC),
\]
so the map $E(i)$ is an equivalence. 
\end{proof}

\begin{lem}\label{lem:kappacompact_gamma}
Let $\kappa$ be a regular and countably closed cardinal.
Then the stable $\infty$-categories $\Gamma \Spt$, $A\Spt$, $I\Spt$ are $\kappa$-compact.
\end{lem}
\begin{proof}
$\mbf{Perf}_{[\mbf{A}^1/\mbf{G}_m]}$ and $\mbf{Perf}_{B\mbf{G}_m}$ have been identified with compact filtered spectra and compact graded spectra respectively, so by Proposition~\ref{prop:Z-complexes}, they are both countable. Furthermore, $\mbf{Perf}_{\mathcal{B}}$ is a pullback of these categories and hence is also countable. The same therefore holds for their subcategories $\mbf{Perf}_{0,[\mbf{A}^1/\mbf{G}_m]}, \mbf{Perf}_{0,\mathcal{B}}$. 
Since $\kappa$ is countably closed, the claim about the $\Cone$-constructions 
follows from Corollary~\ref{cor:catperfk-cofibration-cat}.
\end{proof}

\ssec{Relation to the argument in \cite{KasprowskiWinges}}
It is also possible to give a proof of Theorem~\ref{thm:Gamma_loop} in terms of the categories used in \cite{KasprowskiWinges}.
Using the dictionary provided by Remark~\ref{rem:binacyc}, one can produce the analog of the large diagram in the proof of Theorem~\ref{thm:Gamma_loop}, and repeat the same argument using the results from \cite{KasprowskiWinges}.

The only statement in {\it loc.\ cit.} that uses preservation of filtered colimits is \cite[Lemma~3.10]{KasprowskiWinges}.
This can be proved for arbitrary splitting invariants using the following observation.
Recall that a \emph{universal K-equivalence} is an exact functor $f : \sA \to \sB$ for which there exists an exact functor $g : \sB \to \sA$ such that $[gf] = [\id_\sA] \in K_0(\Funex(\sA,\sA))$ and $[fg] = [\id_\sB] \in K_0(\Funex(\sB,\sB))$.
 In particular, any universal K-equivalence induces an equivalence after applying a splitting invariant. We learned the following lemma from Sasha Efimov:
 
 \begin{lem}\label{lem:associated-graded}
 Let $\sA$ be a stable $\infty$-category which is the union of an ascending sequence
 \[ 0 = \sA_{-1} \subseteq \sA_0 \subseteq \sA_1 \subseteq \sA_2 \subseteq \ldots \]
 of full stable subcategories.
 Assume that each inclusion $i_n : \sA_n \to \sA$ admits a right adjoint functor, and denote the resulting semi-orthogonal decomposition by
 \[\begin{tikzcd}
  \sA_n\ar[r, shift right=2, "i_n"']\ar[r, phantom, "\dashv" {rotate=90}]	& \sA\ar[r, shift right=2, "q_n"']\ar[l, shift right=2, "p_n"']\ar[r, phantom, "\dashv" {rotate=90}] & \sB_n\ar[l, shift right=2, "j_n"'].
 \end{tikzcd}\]
 Then each inclusion functor $f_n : \sA_n \to \sA_{n+1}$, $n \geq -1$, also participates in a semi-orthogonal decomposition
 \[\begin{tikzcd}
	\sA_{n}\ar[r, shift right=2, "f_n"']\ar[r, phantom, "\dashv" {rotate=90}]	& \sA_{n+1}\ar[r, shift right=2, "r_{n+1}"']\ar[l, shift right=2, "g_n"']\ar[r, phantom, "\dashv" {rotate=90}] & \sG_{n+1}\ar[l, shift right=2, "s_{n+1}"'],
 \end{tikzcd}\]
 and the sequence of exact functors $(r_np_n : \sA \to \sB_n)_n$ induces a universal K-equivalence
 \[ \mrm{gr} = (r_np_n)_n : \sA \to \bigoplus_{n \geq 0} \sG_n. \]
\end{lem}
\begin{proof}
 It is straightforward to check that the composite $g_n := p_ni_{n+1} : \sA_{n+1} \to \sA_n$ is right adjoint to the inclusion functor $f_n : \sC_n \to \sC_{n+1}$, giving rise to the second semi-orthogonal decomposition.

 Let $n \geq k \geq 0$ and denote by $f_{k,n} : \sA_k \to \sA_n$ the inclusion functor, so that $i_k \simeq f_{k,n} i_n$.
 If $k < n$, then 
 \[ r_n p_n i_k \simeq r_n p_n i_n f_{k,n} \simeq r_n f_{k,n} \simeq 0 \]
 because $f_{k,n}$ factors over $\sA_{n-1}$.
 Since every object of $\sA$ is of the form $i_k(a)$ for some $k$, the functor $\mrm{gr}$ is well-defined.
 We claim that the exact functor
 \[ t := \sum_{n \geq 0} i_ns_n : \bigoplus_{n \geq 0} \sG_n \to \sC, \]
 witnesses that $\mrm{gr}$ is a universal K-equivalence.
 
 For $n > k \geq 0$, the preceding paragraph shows that $r_np_ni_ks_k \simeq 0$.
 If $n = k$, then $r_np_ni_ns_n \simeq \id_{\sG_n}$.
 Observing that the right adjoint $g_{k,n}$ of $f_{k,n}$ satisfies $p_k \simeq g_{k,n} p_n$, we obtain
 \[ f_k p_k i_n s_n \simeq f_k g_{k,n} p_n i_n s_n \simeq f_k g_{k,n} s_n \simeq f_k g_{k,n-1} g_{n-1} s_n \simeq 0.  \]
 Consequently, the composite $\mrm{gr} \circ i_ks_k : \sG_k \to \bigoplus_{n \geq 0} \sG_n$ is precisely the inclusion of the $k$-th summand.
 This shows that $\mrm{gr} \circ t \simeq \id_{\bigoplus_\mbf{N} \sG_n}$.
 
 Using once more that every object of $\sA$ comes from some filtration step, it follows that for any given $a \in \sA$, we have $q_n(a) \simeq 0$ for all but finitely many $n \geq -1$.
 In particular, the internal direct sum $\bigoplus_{n \geq -1} j_nq_n$ is a well-defined endofunctor of $\sA$.
 Note that there exists a cofiber sequence of exact functors
 \[ \id_\sA \to \bigoplus_{n \geq -1} j_nq_n \to \bigoplus_{n \geq 0} j_nq_n. \]
 Unwinding definitions, one finds that there exist bicartesian squares
 \[\begin{tikzcd}
  i_np_n\ar[r]\ar[d] & i_ns_nr_np_n\ar[d] \\
  \id_\sA\ar[r] & j_{n-1}q_{n-1}
 \end{tikzcd}\]
 which give rise to the cofiber sequence
 \[ t \circ \mrm{gr} \simeq \bigoplus_{n \geq 0} i_ns_nr_np_n \to \bigoplus_{n \geq 0} j_{n-1}q_{n-1} \to \bigoplus_{n \geq 0} j_nq_n. \]
 It follows that $[t \circ \mrm{gr}] = [\id_\sA] \in K_0(\Funex(\sA,\sA))$ as required.
\end{proof} 

%
%
%

\ssec{Applications to products}

Thanks to Theorem~\ref{thm:Gamma_loop} both suspensions and loops of any localizing invariant are categorified via 
certain functors 
\[
\Calk, \Gamma : \Cat^\perf_{\infty} \to \Cat^\perf_{\infty},
\]
i.e.\ there are natural equivalences $E(\Calk(\sC)) \simeq \Sigma E(\sC)$ and $E(\Gamma(\sC)) \simeq \Omega E(\sC)$ 
for {\it any} localizing invariant $E$. 
In particular, somewhat counterintuitively, for a localizing invariant valued in spectra, its zeroth homotopy group 
completely determines its values. 

\begin{cor}\label{cor:pi0_determines_the_invariant}
Let $\sE$ be a stable $\infty$-category with a non-degenerate t-structure and let $E \to F$ be a natural 
transformation  between localizing invariants $\Cat^\perf_{\infty} \to \sE$. 
Assume that $\pi_0E(\sC) \to \pi_0F(\sC)$ is an equivalence for all $\sC \in \Cat^\perf_{\infty}$. Then 
$E(\sC) \to F(\sC)$ is also an equivalence for all $\sC \in  \Cat^\perf_{\infty}$.
\end{cor}
\begin{proof}
Since the t-structure on $\sE$ is non-degenerate, it suffices to show that $\pi_iE(\sC) \to \pi_iF(\sC)$ is an 
equivalence for all $i$. Now we use the natural equivalences $E(\Calk(\sC)) \simeq \Sigma E(\sC)$ and $E(\Gamma(\sC)) \simeq \Omega E(\sC)$. For $i$ negative we have a commutative square
\[
\begin{tikzcd}
\pi_iE(\sC)\arrow[d,"\simeq"] \arrow[r]& \pi_iF(\sC)\arrow[d,"\simeq"]\\
\pi_{0}E(\Calk^{-i}(\sC)) \arrow[r,"\simeq"] &\pi_0F(\Calk^{-i}(\sC)),
\end{tikzcd}
\] 
where the bottom horizontal map is an equivalence by assumption, which implies that the upper map is an equivalence. 
Similarly, for $i$ positive the result follows from commutativity of the square 
\[
\begin{tikzcd}
\pi_iE(\sC)\arrow[d,"\simeq"] \arrow[r]& \pi_{i}F(\sC)\arrow[d,"\simeq"]\\
\pi_{0}E(\Gamma^i\sC) \arrow[r,"\simeq"] &\pi_0F(\Gamma^{i}\sC).
\end{tikzcd}
\]
\end{proof}

Corollary~\ref{cor:pi0_determines_the_invariant} gives an alternative proof of \cite[Theorem~1.3]{KasprowskiWinges}. 

\begin{cor}\label{cor:products}
For any family of stable $\infty$-categories $\sC_i \in \Cat^\perf_\infty$ indexed by a set $I$, the map 
\[
\K(\prod_{i\in I} \sC_i) \to \prod_{i\in I}\K(\sC_i)
\]
is an equivalence
\end{cor}
\begin{proof}
It suffices to prove the claim for a constant family $\sC_i = \sC$ because any other family is a retract of such. 
Define localizing invariants $E$ and $F$ valued in spectra via the following formulas:
\[
E(\sC) = \K(\prod\limits_{i\in I} \sC),
F(\sC) = \prod\limits_{i\in I} \K(\sC).
\]
The natural map $E \to F$ is an isomorphism on $\pi_0$, so by Corollary~\ref{cor:pi0_determines_the_invariant} it is an equivalence.
\end{proof}

\section{Motives over an arbitrary base}\label{app:linear}

In the main body of the paper, we worked in so-called ``absolute'' localizing motives, that is, motives of stable $\infty$-categories. One can, in turn, decide to work over more general bases, for example consider only $\mathbb Z$-linear stable $\infty$-categories (the $\infty$-categorical version of classical dg-categories). More generally, one can work over any base $\mathcal V\in\CAlg(\PrLst)$. To make sense of this precisely, it is easier to work in the dualizable setting rather than in the small setting. From Efimov's recent breakthroughs in the theory of localizing invariants (and the relevant $\mathcal V$-linear analogs, see \cite{ramzi2024dualizablepresentableinftycategories}), the whole theory is insensitive to this and we invite the reader to check for themselves the relevant results. 

The goal of this brief appendix is to indicate the necessary changes in the main text to make the whole story go through $\mathcal V$-linearly.  The point will be that they are quite mild, and except for precise cardinality estimates, one does not lose anything. Therefore, while we do indicate the changes, we do not provide details regarding said changes. 
\begin{notation}
Throughout this section, we fix $\mathcal V$, as well as a regular uncountable cardinal $\lambda$ such that $\mathcal V\in\CAlg(\PrL_\lambda)$. That is, $\mathcal V$ is $\lambda$-compactly generated, its unit is $\lambda$-compact, and tensor products of  $\lambda$-compacts remain $\lambda$-compact. 
\end{notation}
\newcommand{\DblV}{\Mod_\mathcal V^\mathrm{dbl}}
\newcommand{\V}{\mathcal V}

\begin{notation}
We let $\DblV$ denote the $\infty$-category of dualizable $\V$-modules in $\PrL$ and internal left adjoints between them. 
\end{notation}
We refer to \cite{ramzi2024dualizablepresentableinftycategories} for basics about this $\infty$-category. In particular, it has colimits (which are computed underlying in $\PrL$) and so we can define: 
\begin{defn}
A $\V$-localizing invariant is a functor $E: \DblV\to \mathcal E$ for some stable $\infty$-category $\mathcal E$ that sends cofiber sequences along fully faithful functors to cofiber sequences. 

$E$ is called finitary if it preserves filtered colimits. 
\end{defn}
\begin{defn}
A $\V$-motivic equivalence is a map in $\DblV$ which is sent to an equivalence by any finitary $\V$-localizing invariant. 
\end{defn}
\begin{rem}\label{rem:basechangemotivic}
The basechange functor $\V\otimes -: (\PrLst)^{\mathrm{dbl}}\to \DblV$ preserves fully faithful functors and all colimits, therefore precomposition by it sends (finitary) $\V$-localizing invariants to (finitary) localizing invariants. Therefore, $\V\otimes -$ sends motivic equivalences to $\V$-motivic equivalences. 

Similarly, $\V$-motivic equivalences are closed under tensor products. 
\end{rem}
We now wish to prove the following analog of our main result:
\begin{thm}
The localization of $\DblV$ at motivic equivalences is locally small, cocomplete and stable, and is the target of the universal $\V$-localizing invariant. 
\end{thm}
In the case of $\Cat^\perf_\infty$ and ordinary localizing invariants, our proof relied on several ingredients, and we outline them here. 

The first claim is that the direct analog of Proposition~\ref{prop:cat-perf-cofibration-cat} holds in this setting:
\begin{obs}
There is a cofibration structure on $\DblV$ with $\V$-motivic equivalences as weak equivalences and (underlying) fully faithful functors as cofibrations. 
\end{obs}
The only difference with the proof of Proposition~\ref{prop:cat-perf-cofibration-cat} is the construction of 
the functorial factorizations. To explain how to fix this, we first recall that $\V$-modules have underlying $\V$-enriched 
$\infty$-categories, which allows us to take $\V$-enriched presheaf categories of small subcategories of $\V$-modules, which 
we denote by $\mathcal P_\V(\sC)$ if $\sC$ is such a subcategory. With this in hand, we can now recall that by 
\cite[Corollaries~1.7~and~1.59,~Remark~1.60]{ramzi2024dualizablepresentableinftycategories}, there is an internal left adjoint 
embedding $\mathcal M\to \mathcal P_\V(\mathcal M^\lambda)$, natural in the dualizable $\V$-module $\mathcal M$. Since 
$\lambda$ is uncountable, $\mathcal M^\lambda$ admits infinite coproducts and so this presheaf category admits an Eilenberg 
swindle, which implies that it is motivically trivial. All the other constructions and arguments go through unchanged.

The second claim we make is that the set-theoretic care that we took in the body of the paper also goes through here.  The key input to making everything work is the main result of \cite{ramzi2024dualizablepresentableinftycategories}, that the $\infty$-category $\DblV$ is itself presentable. 
In particular, there are two key set-theoretic insights in the main body of the text, which we state here in the $\V$-linear setting. The first was Corollary~\ref{cor:catperfk-cofibration-cat}:

\begin{obs}\label{obs:smallcof}
There is a cofinal class of regular cardinals $\kappa$ such that the cofibration structure described above restricts to a cofibration structure on $\kappa$-compact dualizable $\V$-modules.
\end{obs}
In the absolute case, we singled out the countably closed cardinals (though in fact, with more care, we could have proved it for all regular uncountable cardinals).  The second was the following, used in Lemma~\ref{lem:kappa-localizing} and subsequent results:
\begin{obs}\label{obs:smalllocseq}
There exists a regular cardinal $\kappa$ (in the absolute case, $\kappa= \aleph_1$) so that every localization sequence is a $\kappa$-filtered colimit of localization sequences between $\kappa$-compacts. 
\end{obs}

For the first statement, namely Observation~\ref{obs:smallcof}, we needed to prove that $\kappa$-compacts are closed under cofibers (which is obvious for any $\kappa$), and we needed to check that for suitable $\kappa$'s, the ``mapping cone'' construction from Proposition~\ref{prop:cat-perf-cofibration-cat} (adapted here with $\mathcal P_\V(\mathcal M^\lambda)$ in place of $\Ind(C)^{\aleph_1}$) preserves $\kappa$-compacts --- this was the combination of Corollary~\ref{cor:lax-pullback-compact} and Corollary~\ref{lem:goodisgood}. While we could run similar proofs to get more precise cardinality estimates, we explain the proof here simply using presentability of $\DblV$. For this, we introduce a class of cardinals that corresponds to the countably closed ones in the absolute case. 
\begin{defn}
A regular cardinal $\kappa$ is called $(\V,\lambda)$-closed if it satisfies $\lambda \ll \kappa$ as in \cite[Definition 5.4.2.8]{HTT} and the following conditions hold
\begin{enumerate}
\item[a)]
for all $\lambda$-compact dualizable $\V$-modules $\mathcal M$, $\mathcal P_\V(\mathcal M^\lambda)$ is $\kappa$-compact as a dualizable $\V$-module;
\item[b)] 
for any cospan $(\mathcal M_0\to \mathcal M_{01}\leftarrow \mathcal M_1)$ of $\lambda$-compact dualizable $\V$-modules and internal left adjoints, their oriented pullback $\mathcal M_0\vec{\times}_{\mathcal M_{01}} \mathcal M_1$ is $\kappa$-compact.
\end{enumerate}
\end{defn}
\begin{rem}
Since $\DblV$ is presentable, $(\DblV)^\lambda$ is small and so all $\kappa$'s large enough satisfy parts a) and b) of this definition. Since the first part of the definition is satisfied for cofinally many $\kappa$'s, we also see that there is a cofinal collection of $(\V,\lambda)$-closed cardinals. 
\end{rem}
\begin{lem}
Suppose $\kappa\geq \lambda$ is a $(\V,\lambda)$-closed cardinal, and $\mathcal M$ is a $\kappa$-compact dualizable $\V$-module. Then $\mathcal P_\V(\mathcal M^\lambda)$ is also $\kappa$-compact. 

Furthermore, lax pullbacks of $\kappa$-compact dualizable $\V$-modules are also $\kappa$-compact. 
\end{lem}
\begin{proof}
Both of these follow from the functors $\mathcal P_\V((-)^\lambda)$ and ``lax pullback'' being $\lambda$-accessible: indeed, 
the relation $\lambda \ll\kappa$ guarantees that any $\kappa$-compact dualizable $\V$-module is a $\kappa$-small, 
$\lambda$-filtered colimit of $\lambda$-compacts, and so applying these functors also yields $\kappa$-small, 
$\lambda$-filtered colimits of their values at $\lambda$-compacts, which are by design $\kappa$-compact. 
\end{proof}

For the second statement, Observation~\ref{obs:smalllocseq}, we use again presentability: the $\infty$-category of 
localization sequences in $\DblV$ is presentable and the forgetful functors to $\DblV$ preserve and jointly reflect 
$\kappa$-compacts for $\kappa$ large enough, which proves the claim. 

Finally, to conclude the proof in exactly the same manner as in Section~\ref{section:MotLoc}, we needed a final ingredient involving the Grayson construction:
\begin{obs}
The localization $\DblV[W_{\V-\mathrm{mot}}^{-1}]$ is stable. 
\end{obs} 
To prove this, we note that Remark~\ref{rem:basechangemotivic} gives us a symmetric monoidal, colimit-preserving functor $\Cat^\perf_\infty[W_{\mathrm{mot}}^{-1}]\to \DblV[W_{\V-\mathrm{mot}}^{-1}]$, where the source is already stable by Theorem~\ref{thm:main-intro}, which proves stability of the target as well. 

\begin{exam}
The trace functor as defined in \cite{HSS} always defines a finitary $\V$-localizing invariant $\mathrm{HH}_\V: \DblV\to \V$. In particular, $(\Motnc)_\V = 0$ if and only if $\V=0$. 
\end{exam}

\let\mathbb=\mathbf

{\small
\bibliography{references}

\providecommand{\bysame}{\leavevmode\hbox to3em{\hrulefill}\thinspace}
\providecommand{\MR}{\relax\ifhmode\unskip\space\fi MR }
\providecommand{\MRhref}[2]{%
  \href{http://www.ams.org/mathscinet-getitem?mr=#1}{#2}
}
\providecommand{\href}[2]{#2}
\begin{thebibliography}{BGMN21}
\providecommand{\url}[1]{\href{#1}{{\def~{\textasciitilde}\tt #1}}}

\bibitem[AHI24]{annala2023algebraic}
T.~Annala, M.~Hoyois, and R.~Iwasa, \emph{Algebraic cobordism and a
  Conner–Floyd isomorphism for algebraic K-theory}, Journal of the American
  Mathematical Society \textbf{38} (2024), no.~1, p.~243–289,
  \url{http://dx.doi.org/10.1090/jams/1045}

\bibitem[BGMN21]{polynomial}
C.~Barwick, S.~Glasman, A.~Mathew, and T.~Nikolaus, \emph{K-theory and
  polynomial functors}, preprint \href{http://arxiv.org/abs/2102.00936}{{\sf
  arXiv:2102.00936}}

\bibitem[BGT13]{blumberg2013universal}
A.~J. Blumberg, D.~Gepner, and G.~Tabuada, \emph{A universal characterization
  of higher algebraic {$K$}-theory}, Geom. Topol. \textbf{17} (2013), no.~2,
  pp.~733--838,  \url{https://doi.org/10.2140/gt.2013.17.733}

\bibitem[BM11]{bm:abstracthomotopy}
A.~J. Blumberg and M.~A. Mandell, \emph{Algebraic {{\(K\)}}-theory and abstract
  homotopy theory}, Adv. Math. \textbf{226} (2011), no.~4, pp.~3760--3812,
  \url{https://doi.org/10.1016/j.aim.2010.11.002}

\bibitem[Cis19]{cisinski-rl}
D.-C. Cisinski, \emph{Higher Categories and Homotopical Algebra}, Cambridge
  University Press, Apr 2019,  \url{http://dx.doi.org/10.1017/9781108588737}

\bibitem[CT11]{CisinskiTabuada}
D.-C. Cisinski and G.~Tabuada, \emph{Non-connective K-theory via universal
  invariants}, Compos. Math. \textbf{147} (2011), no.~4, p.~1281–1320,
  \url{http://dx.doi.org/10.1112/S0010437X11005380}

\bibitem[GGN15]{ggn}
D.~Gepner, M.~Groth, and T.~Nikolaus, \emph{Universality of multiplicative
  infinite loop space machines}, Algebr. Geom. Topol. \textbf{15} (2015),
  no.~6, pp.~3107--3153

\bibitem[HSS17]{HSS}
M.~Hoyois, S.~Scherotzke, and N.~Sibilla, \emph{Higher traces, noncommutative
  motives, and the categorified {C}hern character}, Adv. Math. \textbf{309}
  (2017), pp.~97--154,  \url{https://doi.org/10.1016/j.aim.2017.01.008}

\bibitem[KW19]{KasprowskiWinges}
D.~Kasprowski and C.~Winges, \emph{Algebraic K‐theory of stable
  $\infty$‐categories via binary complexes}, J. Topol. \textbf{12} (2019),
  no.~2, p.~442–462,  \url{http://dx.doi.org/10.1112/topo.12093}

\bibitem[Lur17a]{HA}
J.~Lurie, \emph{Higher Algebra}, September 2017,
  \url{http://www.math.harvard.edu/~lurie/papers/HA.pdf}

\bibitem[Lur17b]{HTT}
\bysame, \emph{Higher Topos Theory}, April 2017,
  \url{http://www.math.harvard.edu/~lurie/papers/HTT.pdf}

\bibitem[Lur18]{SAG}
\bysame, \emph{Spectral Algebraic Geometry}, February 2018,
  \url{http://www.math.harvard.edu/~lurie/papers/SAG-rootfile.pdf}

\bibitem[Mou21]{Moulinos_2021}
T.~Moulinos, \emph{The geometry of filtrations}, Bull. Lond. Math. Soc.
  \textbf{53} (2021), no.~5, p.~1486–1499,
  \url{http://dx.doi.org/10.1112/blms.12512}

\bibitem[NS18]{nikolaus-scholze}
T.~Nikolaus and P.~Scholze, \emph{On topological cyclic homology}, Acta Math.
  \textbf{221} (2018), no.~2, pp.~203--409,
  \url{https://doi.org/10.4310/ACTA.2018.v221.n2.a1}

\bibitem[Ram24]{ramzi2024dualizablepresentableinftycategories}
M.~Ramzi, \emph{Dualizable presentable $\infty$-categories}, preprint
  \href{http://arxiv.org/abs/2410.21537}{{\sf arXiv:2410.21537}}

\bibitem[RSW24]{rsw:every-spectrum-is-k-theory}
M.~Ramzi, V.~Sosnilo, and C.~Winges, \emph{Every spectrum is the K-theory of a
  stable $\infty$-category}, preprint
  \href{http://arxiv.org/abs/2401.06510}{{\sf arXiv:2401.06510}}

\bibitem[Sch08]{schlichting_negative_2008}
M.~Schlichting, \emph{Negative {K}-theory of derived categories}, Math. Z.
  \textbf{258} (2008), no.~1, pp.~93--108,
  \url{https://link.springer.com/article/10.1007/s00209-007-0220-2}

\bibitem[Sha17]{ShahThesis}
J.~Shah, \emph{Parametrized higher category theory}, Ph.D. thesis,
  Massachusetts Institute of Technology, 2017, available at
  \url{http://math.mit.edu/~jshah/thesis.pdf}

\bibitem[Wei99]{weiss:hammock}
M.~Weiss, \emph{Hammock localization in {Waldhausen} categories}, J. Pure Appl.
  Algebra \textbf{138} (1999), no.~2, pp.~185--195,
  \url{https://doi.org/10.1016/S0022-4049(98)00009-7}

\end{thebibliography}
}

\parskip 0pt

\end{document}